\documentclass[11pt]{article}
\usepackage{amsmath,fullpage,amssymb,amsthm,hyperref,enumerate}
\usepackage{tikz,color}
 \usetikzlibrary{patterns}
 
\newtheorem{theorem}{Theorem}[section]
\newtheorem{proposition}[theorem]{Proposition}
\newtheorem{lemma}[theorem]{Lemma}
\newtheorem{corollary}[theorem]{Corollary}

\theoremstyle{definition}
\newtheorem{example}[theorem]{Example}  
\newtheorem{definition}[theorem]{Definition}   

\theoremstyle{remark}

\newtheorem{remark}[theorem]{Remark}  

\renewcommand{\P}{\mathcal P}
\newcommand{\Pc}{\overline{\mathcal P}}
\newcommand{\R}{\mathcal R}
\newcommand{\Rc}{\overline{\mathcal R}}
\newcommand{\Ac}{\overline{\mathcal A}}
\newcommand{\G}{\mathcal G}
\newcommand{\V}{\mathcal V}
\newcommand{\D}{\mathcal D}
\newcommand{\A}{\mathcal A}
\newcommand{\B}{\overline{\mathcal A}}
\newcommand{\Q}{\mathcal Q}

\newcommand{\la}{\lambda}
\newcommand{\Z}{\mathbb{Z}}
\DeclareMathOperator\des{des}
\DeclareMathOperator\hdes{\widehat{des}}
\DeclareMathOperator\Des{Des}
\DeclareMathOperator\inv{inv}
\DeclareMathOperator\maj{maj}
\DeclareMathOperator\hmaj{\widehat{maj}}
\DeclareMathOperator\hd{HD}

\newcommand{\qbin}[2]{\begin{bmatrix}{#1}\\ {#2}\end{bmatrix}_q}
\newcommand{\bij}{\theta}
\DeclareMathOperator\dr{dr}
\def\down{-- ++(1,-1) circle(1.2pt)}
\def\up{-- ++(1,1) circle(1.2pt)}
\def\east{-- ++(1,0)}
\def\north{-- ++(0,1)}

\newcommand{\card}[1]{\left|#1\right|}
\newcommand{\Wmn}{{\mathcal W}_{m,n}}
\newcommand{\W}{{\mathcal W}}
\DeclareMathOperator\conj{conj}

\title{Partitions with constrained ranks and lattice paths}

\author{Sylvie Corteel\thanks{IRIF, CNRS et Universit\'e Paris Cit\'e, Paris, France. \texttt{corteel@irif.fr}}
\and
Sergi Elizalde\thanks{Department of Mathematics, Dartmouth College, Hanover, NH 03755. \texttt{sergi.elizalde@dartmouth.edu}}
\and
Carla Savage\thanks{Department of Computer Science, North Carolina State University, Raleigh, NC 27695-8206. \texttt{savage@ncsu.edu}}
}

\begin{document}

\maketitle

\begin{abstract}
In this paper we study partitions whose successive ranks belong to a given set. We enumerate such partitions while keeping track  of the number of parts, the largest part, the side of the Durfee square, and the height of the Durfee rectangle. We also obtain a new bijective proof of a result of Andrews and Bressoud that the number of partitions of~$N$ with all ranks at least $1-\ell$ equals the number of partitions of $N$ with no parts equal to $\ell+1$, for $\ell\ge0$, 
    which allows us to refine it by the above statistics.
Combining Foata's second fundamental transformation for words with Greene and Kleitman's mapping for subsets, interpreted in terms of lattice paths, we obtain enumeration formulas for partitions whose successive ranks satisfy certain constraints, such as being bounded by a constant. 
    \end{abstract}



\section{Introduction}\label{sec:intro}

In this paper we enumerate partitions whose successive ranks satisfy certain conditions. To do so, we consider three combinatorial objects that are related in natural ways: partitions whose Young diagrams fit in an $m \times n$ box, words consisting of $m$ 1s and $n$ 2s, and lattice paths with $m$ up steps and $n$ down steps.
We use Foata's second fundamental transformation~\cite{Foata}, often used to prove the equidistribution of the inversion number and the major index on words, to translate between partitions with rank conditions and lattice paths with constraints on valley heights, which can be enumerated using existing tools.

On the partition side, we adapt a ``minimum rank-raising'' function to a mapping that can be used for bijective proofs of refined counting results for partitions. One the lattice path side, we show that the Greene--Kleitman mapping \cite{GK} in the Boolean lattice translates to a ``lowest valley-lifting'' procedure for lattice paths that can be used for bijective proofs of refined counting results for lattice paths.
Then we show that these two mappings, one on partitions and the other on lattice paths, are essentially the same, with Foata's fundamental transformation provigin the translation between them.

\subsection{Background and Notation}
A partition $\la=(\la_1,\la_2,\dots,\la_k)$ is a weakly decreasing sequence 
of positive integers. The Young diagram of $\la$ is obtained by placing unit squares in $k$ left-justified rows, such that row $i$ from the top consists of $\la_i$ squares for all $i$. We denote its area by $|\la|=\sum_i \la_i$. If $|\la|=N$, we say that $\la$ is a partition of $N$, and write $\la\vdash N$. The elements $\la_i$ are called the parts of $\la$, and $k$ is called the number of parts. Let $\P$ denote the set of all partitions. 

For a given set $S\subseteq\Z$, denote by $\P^S$ the set of all partitions whose parts belong to $S$, and let $\Pc^S=\P\setminus\P^S$. When the elements of $S$ are characterized by a certain inequality, such as $\neq b$ or $\ge b$, we often write that inequality instead of $S$.

The conjugate of $\la$, denoted by $\la'$, is the partition whose parts are $\la'_j=|\{i:\la_i\ge j\}|$, for $1\le j\le \lambda_1$. Its Young diagram is obtained by reflecting the Young diagram of $\la$ along the diagonal.
Let $d(\la)=\max\{i:\la_i\ge i\}$ denote the side of the Durfee square of $\la$, which is the largest square that fits inside the Young diagram of $\la$.
Similarly, let $\dr(\la)=\max\{i:\la_i\ge i+1\}$ denote the height of the Durfee rectangle of $\la$, which is the largest rectangle of height $i$ and width $i+1$, for some $i$, that fits inside the Young diagram of $\la$. 
It is convenient to define the  height of the Durfee rectangle of the empty partition as~0.

For $1\le i\le d(\la)$, the $i$th successive rank of $\la$ is defined to be $r_i(\la)=\la_i-\la'_i$. For $S$ as above, denote by $\R^S$ the set of all partitions whose ranks belong to (or satisfy the condition) $S$, and let $\Rc^S=\P\setminus\R^S$.

Throughout the paper, $m$ and $n$ denote nonnegative integers. If $\Q$ is any set of partitions, denote by $\Q_{m,n}$ the set of partitions in $\Q$ with at most $m$ parts and whose largest part is at most $n$; equivalently, partitions in $\Q$ whose Young diagram fits inside an $m\times n$ rectangle. We abbreviate $\Q_{n,n}$ by $\Q_n$. 
For $N\ge0$, denote by $\Q(N)$ the set of partitions of $N$ in $\Q$. These notations can be combined. For example, $\Rc^{\le b}_{m,n}(N)$ is the set of partitions of $N$ whose Young diagram fits inside an $m\times n$ rectangle, and not all of whose successive ranks are at most $b$.
Note that
$\Q=\bigcup_{N\ge0}\Q(N)$ and  $\Q=\lim_{n\to\infty}\Q_n$.

We use the notation $[n]_q=\frac{1-q^n}{1-q}$, $(q)_n=\prod_{i=1}^n(1-q^i)$, and $$\qbin{n}{k}=\frac{(q)_n}{(q)_k (q)_{n-k}}$$ if $0\le k\le n$, and $\qbin{n}{k}=0$ otherwise.

\subsection{Main Enumerative Results}

In this paper we enumerate partitions whose ranks are bounded by a constant, while keeping track of parameters such as the side of the Durfee square or the height of the Durfee rectangle, and bounding the number of parts and the largest part of the partition. 

When all ranks are bounded below by a nonnegative constant,  we will show how to use Foata's second fundamental transformation to obtain the following formula,
where we define
\begin{equation}
   \label{eq:qtCat}
C_n(q,t)=1 + \frac{1}{[n]_q}\sum_{i=1}^{n-1}t^iq^{i(i+1)}\qbin{n}{i}\qbin{n}{i+1}
\end{equation}
for $n \geq 0$.

\begin{theorem}
\label{thm:lopsided}
For $m,n \geq 0$ and  $-n\le \ell\le 1$,  
\[
\sum_{\la\in \R^{\ge 1-\ell}_{m,n}} t^{d(\la)}q^{|\la|}
=\begin{cases} 
C_{n+\ell}(q,tq^{-\ell})
& \text{if } \ell \leq m-n,\\
\displaystyle\sum_{i\ge0} t^i q^{i(i-\ell)}\left(\qbin{n+\ell}{i}\qbin{m}{i}-\qbin{n+\ell-1}{i-1}\qbin{m+1}{i+1}\right)
& \text{otherwise.}
\end{cases}
\]
\end{theorem}

When negative ranks are allowed, one of the most significant results involving  partitions with constrained ranks is the following generalization of the Rogers--Ramanujan identities, which was discovered and proved by Andrews~\cite{Andrews} for odd $M$ and extended to even $M$ by  Bressoud~\cite{Bressoud80}.

\begin{theorem}[\cite{Andrews,Bressoud80}]
\label{Andrews-Bressoud}
For all integers $r,M,N$ such that $0<r<M/2$,
$$|\R^{[-r+2,M-r-2]}(N)|=|\P^{\not\equiv 0,r,-r\bmod M}(N)|.$$
\end{theorem}

For $r=1$, $M=4$ and $r=2$, $M=5$, Theorem~\ref{Andrews-Bressoud} gives the Rogers--Ramanujan identities.  Letting $M\to\infty$ and setting $r=\ell+1$ in Theorem~\ref{Andrews-Bressoud}, we obtain the following.

\begin{corollary}
\label{cor:AB}
For $\ell\ge0$,
$$|\R^{\ge 1-\ell}(N)|=|\P^{\neq \ell+1}(N)|.$$
\end{corollary}

A bijective proof of this equality was given in~\cite{CSV}, by separating the cases $\ell=0$ and $\ell>0$. A proof via nonintersecting lattice paths was presented in~\cite{BCPS}.

In this paper, we give a new bijective proof of Corollary~\ref{cor:AB} that has 
a uniform description for every $\ell\ge0$, and yields several refinements: it can be restricted to partitions whose Young diagram fits inside a given rectangle, and it allows us to keep track of the side of the Durfee square and the height of the Durfee rectangle. This is achieved by describing a bijection
$$\bij:\Rc_{m,n}^{\ge 1-\ell}(N)\to \P_{m-\ell-1,n+\ell+1}(N-\ell-1)$$
that preserves the Durfee rectangle.  

We further show that the bijection $\bij$ boils down to two famous mappings, which we adapt to partitions and lattice paths: Foata's second fundamental transformation for words~\cite{Foata}, and the Greene--Kleitman mapping between adjacent levels of the Boolean lattice \cite{GK}. The key is to interpret these maps in terms of lattice paths and to prove that the Durfee square of the partition corresponds to the number of descents (valleys) of the path, and that the area of the partition corresponds to the major index of the path.

As a consequence of the bijection $\bij$, we obtain the following formula counting partitions by the height of the Durfee rectangle.
\begin{theorem}
\label{thm:central_drect}
For $m,n,\ell\ge0$ such that $n+\ell\ge m$,
$$
    \sum_{\la\in \R_{m,n}^{\ge 1-\ell}} t^{\dr(\la)}q^{|\la|}=
    \sum_{i\ge0}t^i q^{i(i+1)}\left(\qbin{n-1}{i}\qbin{m+1}{i+1}-q^{\ell+1}\qbin{n+\ell}{i}\qbin{m-\ell}{i+1}\right).
$$
\end{theorem}

Even though the side of the Durfee square is not preserved by $\bij$, it is still possible to use our bijection to obtain a similar formula for this statistic. 

\begin{theorem}
\label{thm:central-simplified}
For $m,n,\ell\ge0$ such that $n+\ell\ge m$, 
\[
\sum_{\la\in \R_{m,n}^{\ge 1-\ell}} t^{d(\la)}q^{|\la|}=
\sum_{i\ge0} t^i q^{i^2}\left(\qbin{n}{i}\qbin{m}{i}-q^\ell\qbin{n+\ell-1}{i-1}\qbin{m-\ell+1}{i+1}\right).
\]
\end{theorem}

An alternative proof of Theorem~\ref{thm:central-simplified} can be obtained from Krattenthaler and Mohanty's work on lattice paths between two boundaries~\cite{KM}. And yet a third proof can be obtained using the techniques from~\cite{BCPS} involving nonintersecting lattice paths, as we will show in Section~\ref{sec:nonintersecting}.

\subsection{Structure of the Paper}

The paper is structured as follows. In Section~\ref{sec:background} we provide some background on words and lattice paths. 
In Section~\ref{sec:Foata}, 
we adapt Foata's classical bijection for words~\cite{Foata} to the context of partitions and lattice paths. 
This bijection allows us to study partitions whose ranks are bounded from below by a constant $1-\ell$ by interpreting them as lattice paths satisfying certain conditions.

For $\ell\le1$, we then apply tools from lattice path enumeration to obtain our formulas. 
This case, where all the ranks of the partition are non-positive (or equivalently non-negative, by conjugation), is treated in Section~\ref{sec:ranks_infinite_interval-lopsided}.
In particular, it will follow that partitions whose ranks are positive are enumerated by refined Catalan numbers.

For $\ell \ge 0$, we describe a bijection between partitions of $N$ with some rank $\le -\ell$ and partitions of $N$ that have some part equal to $\ell+1$. This case is treated in Section~\ref{sec:ranks_infinite_interval-central}. 
Note that the cases $\ell=0$ and $\ell=1$ lie in the intersection where both methods apply.
In Section~\ref{sec:theta} we relate our bijection to a classical construction due to Greene and Kleitman~\cite{GK}, an injection from the set of $k$-element subsets to the set of $(k+1)$-element subsets of $[n]$ (where $k<n/2$), originally used  to give a symmetric chain decomposition of the Boolean lattice.
In Section~\ref{sec:lattice_paths_other}, we 
use nonintersecting paths and tools from~\cite{KM} to
give alternative proofs of some results from Section~\ref{sec:ranks_infinite_interval-central}.
In Section~\ref{sec:final} we briefly discuss the enumeration of partitions whose ranks
lie in an arbitrary set, which is wide open in general.

\section{Background}
\label{sec:background}

\subsection{Words and Partitions}\label{sec:words}

Let $\Wmn$ denote the set of words consisting of $m$ $1$s and $n$ $2$s. For $w=w_1w_2\dots w_{m+n}\in\Wmn$, define its descent set $\Des(w)=\{i\in[m+n-1]:w_i>w_{i+1}\}$, its descent number $\des(w)=\card{\Des(w)}$, its major index $\maj(w)=\sum_{i\in\Des(w)} i$, and its inversion number $\inv(w)=|\{(i,j):i<j,\ w_i>w_j\}|$.

Following~\cite{SaganSavage}, we identify a partition $\la\in\P_{m,n}$ with the word $w\in\Wmn$ obtained by tracing the path from the lower-left corner to the upper-right corner of the $m\times n$ box in which the Young diagram is placed, following the south-east boundary of the diagram, and recording a $1$ for each vertical step and a $2$ for each horizontal step.
The word associated to $\la=(\la_1,\dots,\la_k)$ is
$$w=1^{m-k}2^{\la_k}12^{\la_{k-1}-\la_k}12^{\la_{k-2}-\la_{k-1}}\dots12^{\la_{1}-\la_{2}}12^{n-\la_1},$$
and it satisfies that $\inv(w)=|\la|$. The side of the Durfee square, $d(\la)$,  is the largest $i$ such that, in $w$, the $i$th 2 from the left is to the left of the $i$th 1 from the right. Similarly,
the Durfee rectangle height,
$\dr(\la)$,  is the largest $i$ such that, in $w$, the $(i+1)$st 2 from the left is to the left of the $i$th 1 from the right.

\begin{example}\label{ex:lambda-w}
If $m=4$, $n=6$ and $\lambda=(4,3,3)$, we get
$$
w=1^12^31^12^01^12^11^12^2=1222112122,
$$
$d(\la)=3$ and $\dr(\la)=2$.
\end{example}

Letting $d=d(\la)$, we will also find it useful to describe the word $w$ in terms of $\la$ and $\la'$ as
\begin{equation}
\label{eq:boundary_word}
    w=1^{m-\la_1'}21^{\la_1'-\la'_2}21^{\la_2'-\la'_3}2\ldots1^{\la_{d-1}'-\la'_d}2
    1^{\la_1'-d}2^{\la_1-d}
    1 2^{\la_{d-1}-\la_d}1 2^{\la_{d-2}-\la_{d-1}} \ldots 1 2^{\la_1-\la_2}1 2^{n-\la_1}.
\end{equation}
Continuing Example~\ref{ex:lambda-w}, we have $\lambda'=(3,3,3,1)$ and the boundary word $w$ can be viewed as
$$
w=1^1 2 1^0 2 1^0 2 1^0 2^0 1 2^0 1 2^1 1 2^2 = 1222112122.
$$

\subsection{Lattice paths} \label{sec:lattice_paths}

Denote by $\G_{m,n}$ the set of lattice paths with $m$ up steps $U=(1,1)$ and $n$ down steps $D=(1,-1)$,
starting at the origin and ending at $(m+n,m-n)$. Let $\G_{n}=\G_{n,n}$; these paths are often called {\em grand Dyck paths}.
Words in $\Wmn$ can be viewed as paths in $\G_{m,n}$, by interpreting each $1$ as a $U$ step and each $2$ as a $D$ step. With some abuse of notation, if $P\in\G_{m,n}$ is the path corresponding to a word $w$, we define $\Des(P)$ to equal $\Des(w)$, and similarly for the other statistics. Then $\Des(P)$ equals the set of positions ($x$-coordinates) of the valleys of $P$, $\des(P)$ equals the number of valleys, and $\maj(P)$ equals the sum of the positions of the valleys, where a valley is a vertex preceded by a $D$ step and followed by a $U$ step.

For $P\in\G_{m,n}$, let $v_i(P)$ denote the height of the $i$th valley from the \emph{left} of $P$. 
For a set $S\subseteq\Z$, often abbreviated as an inequality, let 
$$\V^S_{m,n}=\{P\in\G_{m,n}:v_i(P)\in S\ \forall i\},$$
and let $\V^S_n=\V^S_{n,n}$. 

For a fixed $\ell\in\Z$, let $\A^{\ell}_{m,n}$ denote the subset of $\G_{m,n}$ consisting of paths that stay weakly above the line $y=\ell$. Then $\B^{\ell}_{m,n}=\G_{m,n}\setminus\A^{\ell}_{m,n}$ denotes paths that go below this line. Note that $\A^{\ell}_{m,n}=\emptyset$ for $\ell>0$, since the origin is already below the line $y=\ell$ in this case. In particular, $\B^{1}_{m,n}=\G_{m,n}$.

 Notice that $\A^{0}_{n,n}=\V^{\ge0}_n$ is simply the set of {\em Dyck paths} of semilength $n$, which we denote by $\D_n$. It is well known that $|\D_n|=C_n=\frac{1}{n+1}\binom{2n}{n}$, the $n$th Catalan number. When Dyck paths are enumerated with respect to their major index, i.e., the sum of the positions of their valleys, we obtain the following $q$-analogue of Catalan numbers, which goes back to MacMahon:
\begin{equation}\label{eq:qCat}
\sum_{P\in\D_n}q^{\maj(P)}=\frac{1}{[n+1]_q}\qbin{2n}{n}.    
\end{equation}
This is not to be confused with another common $q$-analogue of Catalan numbers due to Carlitz, where $q$ marks the area of Young diagrams that fit inside a staircase shape.

The expression~\eqref{eq:qCat} can be further refined by adding a variable $t$ that marks the number of descents, i.e., valleys. This yields the $q,t$-analogue $C_n(q,t)$ of the Catalan numbers given in equation~\eqref{eq:qtCat}, due to F\"urlinger and Hofbauer~\cite{FH}:
\begin{equation}
    \label{eq:qtCat4D}
C_n(q,t)=\sum_{P\in\D_n}t^{\des(P)}q^{\maj(P)}= 1+ \frac{1}{[n]_q}\sum_{i=1}^{n-1}t^iq^{i(i+1)}\qbin{n}{i}\qbin{n}{i+1}.
\end{equation}
The coefficient of $t^i$ in $C_n(q,t)$ is a $q$-analogue of the well-known Narayana numbers.

More generally, for paths with an arbitrary number of up and down steps, F\"urlinger and Hofbauer prove the following.

\begin{proposition}[{\cite[Equation (4.3)]{FH}}]
\label{prop:FH}
For $m\ge n$,
$$\sum_{P\in\A^{0}_{m,n}} t^{\des(P)} q^{\maj(P)} = \sum_{i\ge0} t^i q^{i^2} \left(\qbin{m}{i}\qbin{n}{i} -\qbin{m-1}{i-1}\qbin{n+1}{i+1}\right).$$
\end{proposition}
Thus we have 
\begin{equation}
\label{Ac0}
\sum_{P\in\Ac^{0}_{m,n}} t^{\des(P)} q^{\maj(P)} =
\sum_{i\ge0} t^i q^{i^2} \qbin{m-1}{i-1}\qbin{n+1}{i+1}.
\end{equation}
Note that if $m \geq n$, then 
$\V^{\ge0}_{m,n}= \A^{0}_{m,n}$,
a fact we will sometimes exploit in order to apply Proposition~\ref{prop:FH}.

\section{Foata's bijection}\label{sec:Foata}

In~\cite{Foata}, Foata introduced a bijection $\phi$ for words over the positive integers, which preserves the number of occurrences of each integer, and has the property that $\maj(w)=\inv(\phi(w))$ for every word $w$. The map $\phi$ is known as Foata's {\em second fundamental transformation}.

In the special case of words over $\{1,2\}$, it was shown in \cite{CSV} that $\phi$ has the simple recursive description
$$\phi(w2) = \phi(w)2; \ \ \ \ \phi(w11) = 1 \phi(w1); \ \ \ \ \phi(w21) = 2\phi(w) 1,$$
with initial conditions $\phi(1)=1$ and $\phi(\epsilon)=\epsilon$.

Noting that for positive $m,n$, any $w \in \Wmn$ with at least one descent can be written uniquely as $w=x21^c2^a$, where $c>0$, $a \geq 0$, we have
\begin{equation}\label{eq:phi-ca}\phi(x21^c2^a) = 1^{c-1}2\phi(x)12^a.\end{equation}
Unwinding this recursive definition of $\phi$,  write $w \in \Wmn$ in the form $w=1^{m_0}2^{n_0}1^{m_1}2^{n_1} \ldots 1^{m_d}2^{n_d}$, where $m_0\geq 0$, $n_d \geq 0$, and otherwise all $m_i,n_i$ are positive.  Then $\des(w) = d$ and
\begin{equation}
\label{foata_unwound}
  \phi(1^{m_0}2^{n_0}1^{m_1}2^{n_1} \ldots 1^{m_d}2^{n_d}) =
1^{m_d-1}2 1^{m_{d-1}-1}2 \ldots 1^{m_1-1}2 1^{m_0}2^{n_0-1}12^{n_1-1}1 \ldots 2^{n_{d-1}-1}12^{n_d}.  
\end{equation}

\subsection{Foata's bijection in terms of partitions}

As described in Section~\ref{sec:words}, words in $\Wmn$ encode the boundaries of Young diagrams of partitions in $\P_{m,n}$, as well as lattice paths in $\G_{m,n}$.
By interpreting the domain of $\phi^{-1}:\Wmn\to\Wmn$ as partitions and the codomain as lattice paths, we view this map as a bijection
$$\phi^{-1}:\P_{m,n}\to \G_{m,n}.$$

Given nonempty $\la \in \P_{m,n}$, let $z$ be the word in $\Wmn$ encoding the boundary of $\la$. For $m,n>0$, we can write $z$ uniquely as
$$z=1^b2u12^a,$$
where $a=m-\la_1'\ge0$, $b=n-\la_1\ge0$ and $u \in \P_{m-b-1,n-a-1}$.
Note that $u$ is the word describing the boundary of the partition obtained from $\la$ by removing the outermost hook of $\la$, consisting of the $(m-b)+(n-a)-1$ squares in the first column and first row of $\la$.

Let $w=\phi^{-1}(z)$. By equation~\eqref{eq:phi-ca}, we have
$$ w = \phi^{-1}(1^b2u12^a) = \phi^{-1}(u)21^{b+1}2^a.$$

Using the hooks of $\la$, we can provide a visual description of $\phi^{-1}(\la)$, from which some of its properties will be more apparent (see Figure \ref{fig:phi-1}).

\begin{definition}
\label{def:hook_decomp}
Let $\la$ be a partition with $d(\la)=d$. The hook decomposition of $\la$ is
$$\hd(\la)= \{h_1, h_2, \ldots, h_d\},$$
where, for $1 \leq j \leq d$, we let
\begin{align*}
h_j&= a_j + b_j,\\
 a_{d+1-j}&=\la_j-j+1, \\
b_{d+1-j}& = \la'_j -j.
\end{align*}
\end{definition}

\begin{lemma}\label{Foata_to_hooks} Let $\la$ be a partition with hook decomposition as in Definition \ref{def:hook_decomp}.  Then
 $P = \phi^{-1}(\la)$ is the unique path in $\G_{m,n}$ with $d(\la)$ valleys whose $j$th valley from the left, for $1 \leq j \leq d(\la)$, is at the point
\begin{equation*}
    V_j  =  (b_j+a_j,b_j-a_j).
\end{equation*}
\end{lemma}
\begin{proof}

The unique path $P$ in $\G_{m,n}$ whose $j$th valley is at $V_j$ for each $1 \leq j \leq d$ is described by the word
\begin{equation*}
w=1^{b_1} 2^{a_1} 1^{b_2-b_1} 2^{a_2-a_1} 1^{b_3-b_2} 2^{a_3-a_2}\ldots 1^{b_d-b_{d-1}} 2^{a_d-a_{d-1}}
1^{m-b_d} 2^{n-a_d}.
\end{equation*}
By equation \eqref{foata_unwound}, applying $\phi$ to $w$ gives the word
\begin{equation*}
\phi(w)=1^{m-b_d-1} 2  1^{b_d-b_{d-1}-1} 2 \ldots 1^{b_2-b_1-1} 2
1^{b_1} 2^{a_1-1}
1 2^{a_2-a_1-1} 1 2^{a_3-a_2-1} \ldots 1 2^{a_d-a_{d-1}-1}
1 2^{n-a_d}.
\end{equation*}
Let $\la$ be the partition whose boundary is $\phi(w)$.
By equation \eqref{eq:boundary_word}, matching exponents,
\begin{align*}
m-\la_1' &=m-b_d-1,\\
\la'_j-\la'_{j+1} & =  b_{d+1-j} - b_{d-j}-1 \  {\rm for} \  1 \leq j \leq d-1,\\
\la'_d-d & =  b_1,\\
\la_d-d & =  a_1-1,\\
\la_j-\la_{j+1} & =  a_{d+1-j} - a_{d-j}-1 \  {\rm for} \  d-1 \geq j \geq 1,\\
n-\la_1 & =  n-a_d.
\end{align*}
This proves the lemma for $j=1$ and $j=d$.  For $2 \leq j \leq d-1$,
$$ b_j-b_1 =  \sum_{i=2}^j(b_i-b_{i-1}) =  \sum_{i=2}^j \left(1+\la'_{d+1-i}-\la'_{d+1-(i-1)}\right) =  (j-1) + \la'_{d+1-j}-\la'_d.$$
Since $b_1 = \la'_d-d$, this gives $b_j = \la'_{d+1-j}-(d+1-j)$.
It can be shown similarly that $a_i=\la_{d+1-j}-(d+1-j) -1$ for $2 \leq j \leq d-1$.
\end{proof}

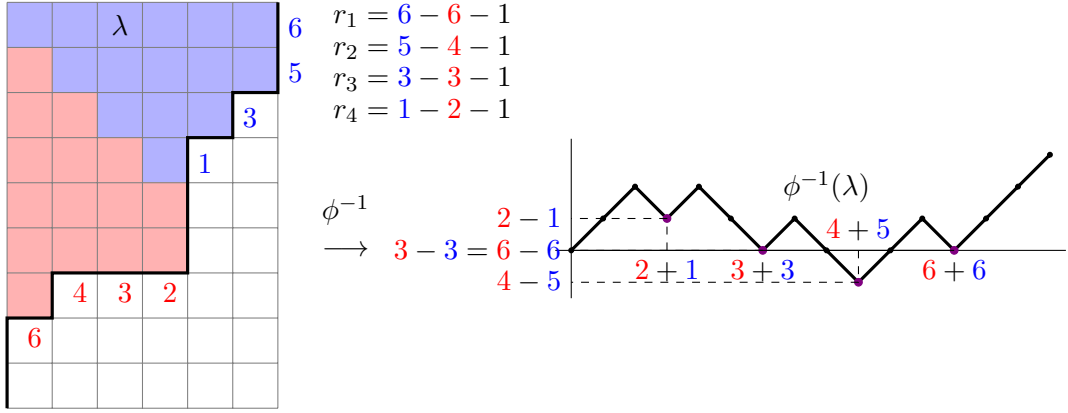
\begin{figure}[htb]
\centering
    \begin{tikzpicture}[scale=0.6]
    \fill[blue!30] (0,8) rectangle (6,9);
    \fill[blue!30] (1,7) rectangle (6,8);
    \fill[blue!30] (2,6) rectangle (5,7);
    \fill[blue!30] (3,5) rectangle (4,6);
    \fill[red!30] (0,8) rectangle (1,2);
    \fill[red!30] (1,7) rectangle (2,3);
    \fill[red!30] (2,6) rectangle (3,3);
    \fill[red!30] (3,5) rectangle (4,3);
      \draw[gray,thin] (0,0) grid (6,9);
      \draw[very thick] (0,0) coordinate(d0)
      -- ++(0,1) coordinate(d1)
      -- ++(0,1) coordinate(d2)
      -- ++(1,0) coordinate(d3)
      -- ++(0,1) coordinate(d4)
      -- ++(1,0) coordinate(d5)
      -- ++(1,0) coordinate(d6)
      -- ++(1,0)  coordinate(d7)
      -- ++(0,1) coordinate(d8)
      -- ++(0,1) coordinate(d9)
      -- ++(0,1)  coordinate(d10)
      -- ++(1,0)  coordinate(d11)
      -- ++(0,1) coordinate(d12)
      -- ++(1,0) coordinate(d13)
      -- ++(0,1) coordinate(d14)
       -- ++(0,1) coordinate(d15);
      \draw[red] (d3) node[below left] { $6$};
      \draw[red] (d5) node[below left] { $4$};
      \draw[red] (d6) node[below left] { $3$};
      \draw[red] (d7) node[below left] { $2$};
      \draw[blue] (d9) node[above right] { $1$};
      \draw[blue] (d11) node[above right] { $3$};
        \draw[blue] (d13) node[above right] { $5$};    
        \draw[blue] (d14) node[above right] { $6$};
     \draw (2.5,8.5) node {$\la$};
      \draw (7.5,3.5) node [label=$\phi^{-1}$] {$\longrightarrow$};
      \draw (7,8.7) node[right] {$r_1={\color{blue}6}-{\color{red}6}-1$};
      \draw (7,8) node[right] {$r_2={\color{blue}5}-{\color{red}4}-1$};
      \draw (7,7.3) node[right] {$r_3={\color{blue}3}-{\color{red}3}-1$};
      \draw (7,6.6) node[right] {$r_4={\color{blue}1}-{\color{red}2}-1$};
     
      \begin{scope}[shift={(12.5,3.5)},scale=.7071]
       \draw (-.5,0)--(15.5,0);
       \draw (0,-1.5)--(0,3.5);
     \draw[very thick](0,0) circle(1.2pt)\up\up\down\up\down\down\up\down\down\up\up\down\up\up\up; 
     \draw (8,2) node {$\phi^{-1}(\la)$};
       \filldraw[violet] (3,1) circle (.12);
      \filldraw[violet] (6,0) circle (.12);
      \filldraw[violet] (9,-1) circle (.12);
      \filldraw[violet] (12,0) circle (.12);
      \draw[dashed] (3,0) node[below] {${\color{red}2}+{\color{blue}1}$} --(3,1)--(0,1) node[left]{${\color{red}2}-{\color{blue}1}$};
     \draw[dashed] (6,0) node[below] {${\color{red}3}+{\color{blue}3}$} --(6,0)--(0,0) node[left]{${\color{red}3}-{\color{blue}3}={\color{red}6}-{\color{blue}6}$};
    \draw[dashed] (9,0) node[above] {${\color{red}4}+{\color{blue}5}$} --(9,-1)--(0,-1)
    node[left]{${\color{red}4}-{\color{blue}5}$};
    \draw (12,0) node[below] {${\color{red}6}+{\color{blue}6}$}--(0,0);
      \end{scope}
    \end{tikzpicture}
  \caption{The map $\phi^{-1}:\P_{9,6}\to \G_{9,6}$ applied to the partition $\la=(6,6,5,4,4,4,1)$. The sequences $a=(1,3,5,6)$ and $b=(2,3,4,6)$ are used to determine the coordinates of the valleys of $\phi^{-1}(\la)$.}\label{fig:phi-1}
\end{figure}

We summarize the statistics preserved by $\phi^{-1}$ in the following lemma. Part (a) is a well-known property of Foata's bijection, and part (b) is proved in \cite[Cor.\ 2.4]{SaganSavage}. Part (c), in terms of a related map $\psi$, is proved in~\cite{BBES14} for the case $m=n$, and in~\cite{BBES16} for the general case. 
Part (d) is proved in~\cite{ElizaldeFPSAC} for the case $m=n$. Part (e) appears to be new. We include proofs below to make the paper self-contained.

\begin{lemma}\label{lem:Foata}
Let $\la\in\P_{m,n}$ and $P=\phi^{-1}(\la)\in\G_{m,n}$. Then

\begin{enumerate}[(a)]
    \item $|\la|=\maj(P)$, 
    \item $d(\la)=\des(P)$, 
    \item $\hd(\la)=\Des(P)$, 
    \item $r_i(\la)=-1-v_{d+1-i}(P)$ for all $i$,
    \item $\dr(\la')=\begin{cases}
\des(P)-1, & \text{if $P$ starts with a $D$};\\
\des(P), & \text{otherwise}.
\end{cases}
$
\end{enumerate}
\end{lemma}

\begin{proof}
The descents of $P$ are the $x$-coordinates of its valleys.  Thus by Lemma \ref{Foata_to_hooks}, $\des(P)=d(\la)$ and  $\Des(P)=\hd(\la)$, proving parts (b) and (c).  Noting that $|\la|$ equals the sum of the entries in $\hd(\la)$, which is $\maj(P)$, we deduce part (a).  Recall that $v_i$ is the height of the $i$th valley from the left, which is at the point  
$V_{i}=(b_{i}+a_{i},b_{i}-a_{i})$, and so  
$v_{d+1-i}  = b_{d+1-i}-a_{d+1-i} = \la_j'-\la_j-1 = -r_i(\la)-1$, proving part (d).

Finally, to prove part (e), first note that $\dr(\la')$ is the largest $i$ such that the Young diagram of $\la$ contains an $(i+1)\times i$ rectangle. It follows that $\dr(\la')=d(\la)-1$ if $\la_d'=d$, and $\dr(\la')=d(\la)$ otherwise.
The condition $\la_d'=d$ is equivalent to $b_1=0$, and to the fact that the leftmost valley of $P$ has coordinates $(a_1,-a_1)$. Since $a_1 \geq 1$, this is equivalent to $P$ starting with a $D$ step. Part (e) now follows using that $d(\la)=\des(P)$.
\end{proof}

The first equality in the next lemma follows immediately from Lemma~\ref{lem:Foata}(a). The second equality is the well-known interpretation of the $q$-binomial coefficients as counting partitions inside a rectangle with respect to their area.

\begin{lemma}\label{lem:qbin}
$$\sum_{P\in\G_{m,n}}q^{\maj(P)}
=\sum_{\la\in\P_{m,n}}q^{|\la|}
=\qbin{m+n}{n}.$$
\end{lemma}

Similarly, using Lemma~\ref{lem:Foata}(b), we can refine this formula by the number of descents. For the second equality below, we separate the Young diagram of the partition into its Durfee square $i\times i$ plus a partition in $\P_{i,n-i}$ to its right and the conjugate of a partition in $\P_{i,m-i}$ below it. 

\begin{lemma}\label{lem:qbin2}
$$\sum_{P\in\G_{m,n}}t^{\des(P)}q^{\maj(P)}
=\sum_{\la\in\P_{m,n}} t^{d(\la)} q^{|\la|} 
=\sum_{i\ge0} t^i q^{i^2}\qbin{n}{i}\qbin{m}{i}.$$
\end{lemma}

\begin{remark}  
 Our description of $\phi^{-1}$ in Lemma \ref{Foata_to_hooks} is inspired by the bijection $\psi:\P_n\to \G_n$ that was introduced in \cite[Lem.\ 3.5]{BBES14} to study descents on pattern-avoiding involutions, later generalized to a bijection $\psi:\P_{m,n}\to\G_{m,n}$ in \cite{BBES16} (specifically, $\psi$ is the inverse of the bijection $g$ is described in \cite[Lem.\ 10]{BBES16}), and recently used in \cite{ElizaldeFPSAC} to measure the symmetry of partitions. Despite their superficially different definitions, it turns out that this bijection $\psi$ and Foata's bijection restricted to partitions are related by $$\psi(\la)=\overline{\phi^{-1}(\la')},$$ where  $\overline{P}$ denotes the path obtained from $P$ by reflecting along the $x$-axis, i.e., switching $U$ and $D$ steps.
\end{remark}

\subsection{Restricting Foata's bijection}

Lemma~\ref{lem:Foata}(d) allows us to use $\phi^{-1}$ to obtain bijections between partitions with constrained ranks and paths with valleys at constrained heights. Specifically, for any set $S\subseteq\Z$, we obtain a bijection
\begin{equation}\label{eq:phiS} \phi^{-1}:\R^S_{m,n}\to \V^{-S-1}_{m,n}, \end{equation}
where $-S-1=\{-s-1:s\in S\}$. It follows that 
\begin{equation}\label{eq:RV}
\sum_{\la\in\R^S_{m,n}} t^{d(\la)}q^{|\la|}=\sum_{P\in\V^{-S-1}_{m,n}}t^{\des(P)}q^{\maj(P)}.
\end{equation}
Setting $S=(-\infty,\ell-1]$ in equation~\eqref{eq:phiS} gives a bijection 
\begin{equation}\label{eq:phib}
\phi^{-1}:\R^{\le \ell-1}_{m,n}\to\V^{\ge -\ell}_{m,n}.
\end{equation}

Since conjugation preserves the Durfee square and the area of a partition, combining conjugation with the bijection in equation~\eqref{eq:phiS}, we can make the following observation about lattice paths, which appears to be new.

\begin{proposition}\label{prop:flipvalleys}
Let $S\subseteq\Z$ and let $m,n\ge0$. There is a $(\des,\maj)$-preserving bijection $$\V^{-S-1}_{m,n}\to\V^{S-1}_{n,m}$$
given by $P\mapsto \phi^{-1}(\phi(P)')$.  
\end{proposition}

\begin{example}\label{ex:Dyck}
Taking $S=\Z_{>0}$ and $m=n$ in Proposition~\ref{prop:flipvalleys}, we get a bijection between the set $\V^{\le -2}_n$ of grand Dyck paths with all valleys on or below the line $y=-2$ and the set $\V^{\ge0 }_{n}=\D_n$ of Dyck paths.
Note that another bijection between these two sets can be obtained by reflecting each path $P\in\V^{\le -2}_n$ along the $x$-axis (resulting in a grand Dyck path with all peaks on or above the line $y=2$), and then replacing each block that lies below the $x$-axis, which must be of the form $D^iU^i$ for some $i$, with $(UD)^i$. This produces a Dyck path, and it is easy to see that this process is reversible; however, this bijection does not preserve $\des$ or $\maj$,
unlike the one given by Proposition~\ref{prop:flipvalleys}.
\end{example}

\section{Partitions with all ranks at least $1-\ell$ for $\ell\le 1$} 
\label{sec:ranks_infinite_interval-lopsided}

\subsection{Refined enumeration of partitions inside a rectangle}

Next we apply the results of the previous section to prove Theorem~\ref{thm:lopsided}, enumerating partitions with all ranks at least $1-\ell$, where $\ell \leq 1$.
\begin{proof}[Proof of Theorem~\ref{thm:lopsided}]
Applying conjugation followed by the bijection in equation~\eqref{eq:phib}, we obtain
from equation \eqref{eq:RV}
$$\sum_{\la\in \R^{\ge 1-\ell}_{m,n}} t^{d(\la)}q^{|\la|}=\sum_{\la\in \R^{\le \ell-1}_{n,m}} t^{d(\la)}q^{|\la|}=\sum_{P\in\V^{\ge -\ell}_{n,m}}t^{\des(P)}q^{\maj(P)}.$$
Thus, it suffices to enumerate paths in $\V^{\ge -\ell}_{n,m}$ with respect to the statistics $\des$ and $\maj$. 

In the case $\ell=1$, there is a simple bijection between $\V^{\ge-1}_{n,m}$ and $\V^{\ge0}_{n+1,m}$, obtained by adding a $U$ step at the beginning of each path.
This bijection increases the positions of all the valleys by $1$.

For $\ell\le 0$, any path in $\V^{\ge -\ell}_{n,m}$ must begin with $-\ell$ $U$ steps. 
Thus, there is also a simple bijection between $\V^{\ge -\ell}_{n,m}$ and $\V^{\ge0}_{n+\ell,m}$, this time obtained by removing the prefix $U^{-\ell}$ from each path. This bijection shifts the positions of all the valleys $-\ell$ units to the left.

Combining both cases, it follows that, for all $\ell\le1$, 
\begin{equation}\label{eq:Vb-1}
\sum_{P\in\V^{\ge -\ell}_{n,m}}t^{\des(P)}q^{\maj(P)}=\sum_{P\in\V^{\ge0}_{n+\ell,m}}(tq^{-\ell})^{\des(P)}q^{\maj(P)}.
\end{equation} 

The right endpoint of any path in $\V^{\ge0}_{n+\ell,m}$ is at height $n+\ell-m$. Let us consider two cases. 

If $n+\ell-m\ge 0$, then $\V^{\ge0}_{n+\ell,m}=\A^0_{n+\ell,m}$. 
In this case, equation~\eqref{eq:Vb-1}, together with Proposition~\ref{prop:FH}, implies that
\begin{align*}
\sum_{P\in\V^{\ge -\ell}_{n,m}}t^{\des(P)}q^{\maj(P)}&=\sum_{P\in\A^{0}_{n+\ell,m}}(tq^{-\ell})^{\des(P)}q^{\maj(P)}\\
&=\sum_{i\ge0} t^i q^{i(i-\ell)} \left(\qbin{n+\ell}{i}\qbin{m}{i} -\qbin{n-1+\ell}{i-1}\qbin{m+1}{i+1}\right).
\end{align*}

If $n+\ell-m\le 0$, paths in $\V^{\ge0}_{n+\ell,m}$ must end with $D^{m-n-\ell}$, since otherwise they would have valleys below $y=0$. These $D$ steps at the end do not contribute to $\des$ or $\maj$, so we can remove them to obtain a $(\des,\maj)$-preserving bijection between 
$\V^{\ge0}_{n+\ell,m}$ and $\D_{n+\ell}$.
Thus, in this case, equations~\eqref{eq:Vb-1} and~\eqref{eq:qtCat4D} imply that
\[
\sum_{P\in\V^{\ge -\ell}_{n,m}}t^{\des(P)}q^{\maj(P)}=\sum_{P\in\D_{n+\ell}}(tq^{-\ell})^{\des(P)}q^{\maj(P)}=C_{n+\ell}(q,tq^{-\ell}).\qedhere
\]
\end{proof}

Setting $m=n$ in Theorem~\ref{thm:lopsided}, we get the following.

\begin{corollary}\label{cor:lopsided}
For $n\ge0$ and $-n\le \ell\le 1$,
$$\sum_{\la\in \R^{\ge 1-\ell}_n} t^{d(\la)}q^{|\la|}
=C_{n+\ell}(q,tq^{-\ell}).$$
\end{corollary}

\begin{example}\label{ex:R<0bij}
Setting $\ell=0$, which is equivalent to constraining the ranks to be positive, 
equation~\eqref{eq:phib} gives a bijection between $\R^{<0}_n$ and $V^{\ge0}_n=\D_n$. Conjugation gives a bijection between $\R^{<0}_n$ and $\R^{>0}_n$, and Corollary~\ref{cor:lopsided} simply states that
\begin{equation}\label{eq:R<0}
\sum_{\la\in \R^{>0}_n} t^{d(\la)}q^{|\la|}=\sum_{P\in\D_n}t^{\des(P)}q^{\maj(P)}
=C_n(q,t),
\end{equation}
recovering \cite[Thm.\ 3.1]{SaganSavage}. Combining equation~\eqref{eq:R<0} with Example~\ref{ex:Dyck}, we obtain the formula 
$$\sum_{P\in\V^{\le-2}_n}t^{\des(P)}q^{\maj(P)}=C_n(q,t).$$
which appears to be new.
\end{example}

\subsection{Removing the bounding box}

In Theorem~\ref{thm:lopsided} and Corollary~\ref{cor:lopsided}, partitions are restricted to be inside a box. Next we remove this requirement by taking the limit as the side lengths of the box go to infinity.

An interesting fact about the $q$-Catalan numbers from equation~\eqref{eq:qCat} is that their limit as $n$ goes to infinity gives a valid formal power series:
\begin{equation}\label{eq:limCq1}\lim_{n\to\infty} C_n(q,1)=
\lim_{n\to\infty} \frac{(1-q^{2n})(1-q^{2n-1})\dots(1-q^{n+2})}{(1-q^n)(1-q^{n-1})\dots(1-q^2)}=
\prod_{i\ge2}\frac{1}{1-q^i},\end{equation}
and similarly for the $q,t$-Catalan numbers from  equation~\eqref{eq:qtCat}:
$$\lim_{n\to\infty} C_n(q,t)
=1+\sum_{i\ge1}\frac{t^i q^{i(i+1)}}{(1-q)(1-q^2)^2(1-q^3)^2\dots(1-q^i)^2(1-q^{i+1})}.$$

Taking the limit as $n$ goes to infinity in Corollary~\ref{cor:lopsided} and setting $\ell=1-b$ we get the following.

\begin{corollary}\label{cor:lopsidedlimit}
For $b\ge0$,
$$
\sum_{\la\in \R^{\ge b}} t^{d(\la)}q^{|\la|}
=1+\sum_{i\ge1}\frac{t^i q^{i(i+b)}}{(1-q)(1-q^2)^2(1-q^3)^2\dots(1-q^i)^2(1-q^{i+1})}.
$$
\end{corollary}

There is also a direct proof of this corollary. Indeed, note that the factor  $$\frac{1}{(1-q)(1-q^2)^2(1-q^3)^2\dots(1-q^i)^2(1-q^{i+1})}$$
is the generating function for plane partitions of shape $(i,i)$. Such plane partitions can be interpreted as a pair of partitions $(\mu,\nu)$, each with at most $i$ parts, where $\mu_j\ge\nu_j$ for $1\le j\le i$. In the right-hand side of Corollary~\ref{cor:lopsidedlimit}, we can interpret  $t^iq^{i(i+b)}$ as placing an $i\times (i+b)$ rectangle in the upper-left corner of the Young diagram, to which we attach the Young diagram of $\mu$ to its right, and the Young diagram of $\nu'$ below. This construction yields a partition with all ranks at least $b$, and Durfee square of side $i$.

For $b=1$, the expression on the right-hand side in  Corollary~\ref{cor:lopsidedlimit} can be interpreted as enumerating partitions with no parts of size $1$, according to the height $i$ of the Durfee rectangle. To an $i\times (i+1)$ rectangle, we attach the Young diagram of a partition with at most $i$ parts 
(contributing $\frac{1}{(1-q)(1-q^2)(1-q^3)\dots(1-q^i)}$)
to its right, and the Young diagram of a partition with largest part at most $i+1$ and no parts equal to $1$ (contributing $\frac{1}{(1-q^2)(1-q^3)\dots(1-q^i)(1-q^{i+1})}$) below this rectangle.
Thus, considering that for $\la\in\R^{>0}$ we have $d(\la)=\dr(\la)$, we recover the fact, proved bijectively in~\cite{CSV}, that the height of the Durfee rectangle is equidistributed on partitions with positive ranks and partitions with no parts equal to $1$.

If we set $t=1$ in Corollary~\ref{cor:lopsidedlimit}, we recover Corollary~\ref{cor:AB} for $b=1$, or equivalently $\ell=0$. 
This can also be seen directly by setting $t=1$ in equation~\eqref{eq:R<0} and taking the limit as $n$ goes to infinity using equation~\eqref{eq:limCq1}.

For $b=1-\ell=0$, Corollary~\ref{cor:AB} tells us that the number of partitions of $N$ with non-negative ranks equals the number of partitions of $N$ with no parts equal to 2. Thus, with the specialization $t=1$, the right-hand side of Corollary~\ref{cor:lopsidedlimit} for $b=0$ should be the generating function for partitions with no part equal to 2, that is,
\begin{equation}\label{eq:limCn}
\lim_{n\to\infty}C_n(q,q^{-1})=1+\sum_{i\ge1}\frac{q^{i^2}}{(1-q)(1-q^2)^2(1-q^3)^2\cdots(1-q^i)^2(1-q^{i+1})}=\prod_{\substack{i\ge1\\ i\neq 2}} \frac{1}{1-q^i}.
\end{equation}

To see this directly, write the sum in equation~\eqref{eq:limCn} as
\begin{align*}
1+\sum_{i\ge1}\frac{q^{i^2}}{(1-q)(1-q^2)^2(1-q^3)^2\cdots(1-q^i)^2}
+\sum_{i\ge1}\frac{q^{i^2+i+1}}{(1-q)(1-q^2)^2(1-q^3)^2\cdots(1-q^i)^2(1-q^{i+1})}
\end{align*}
by multiplying each summand by $1-q^{i+1}+q^{i+1}$. 
Each summand in the first sum can be interpreted as counting partitions with Durfee square of side $i$ having no twos and an even number of ones. Each summand in the second sum  corresponds to partitions with Durfee rectangle of height $i$ having no twos and an odd number of ones. Indeed, the contribution $q^{i^2+i+1}$ comes from a Durfee rectangle $i\times (i+1)$ and a single part equal to 1.

Interestingly, for $b\ge2$, the set $\R^{\ge b}$ is not covered by Corollary~\ref{cor:AB}, but Corollary~\ref{cor:lopsidedlimit} still applies.

\begin{remark} For $b\ge2$, the right-hand side of Corollary~\ref{cor:lopsidedlimit} (even with the specialization $t=1$) does not seem to have an interpretation in terms of partitions with forbidden parts. For example, for $b=2$, the product formula for this right-hand side is 
$$
\frac{1}{\prod_{i\ge 1}(1-q^i)^{s_i}}
$$
where $$s=(0,0,1,1,2,1,2,1,1,0,1,0,1,1,1,2,2,2,1,1,-1,0,-1,1,0,3,3,4,3,3\ldots).$$
\end{remark}

\section{Partitions with all ranks at least $1-\ell$ for $\ell \ge 0 $} 
\label{sec:ranks_infinite_interval-central}

We know by Corollary~\ref{cor:AB} that, for $\ell\ge 0$, the number of partitions of $N$ where all ranks are at least $1-\ell$ equals the number of partitions of $N$ with no parts equal to $\ell+1$. In~\cite{CSV}, Corteel, Savage and Venkatraman gave a bijective proof of this fact.

In order to prove a refinement,  we take a different view by setting up a
 bijection between the complements of the sets, that is, between $\Rc^{\ge 1-\ell}(N)$ and $\Pc^{\neq \ell+1}(N)$.
Note that $\Rc^{\ge 1-\ell}(N)$ is the set of partitions of $N$ that have {\em some} rank $\le -\ell$, and $\Pc^{\neq \ell+1}(N)$ is the set of partitions of $N$ that have {\em some} part equal to $\ell+1$. The second set is obviously in bijection with $\P(N-\ell-1)$, by simply removing a part equal to $\ell+1$.

Our bijection has some advantages over the bijection in~\cite{CSV}:
    it can be refined by restricting the partitions to be inside a rectangle, and
it yields a refinement that keeps track of the side of the Durfee square.

\subsection{The bijections $f$ and $\bij$}\label{sec:f_bij}

As in \cite{CSV}, we will make use of a minimum-rank-increasing map, which we denote by $f$,  for partitions in $\Rc^{\ge 1}$. 

Given $\la\in\Rc^{\ge 1}$, let $\tau=\tau(\la)$ be the minimum rank of $\la$ and let $i=i(\la)$ be the largest index such that $r_i(\la) = \tau$. First, note that
since $\tau \leq 0$, $\la$ must have a part of size $i$. Indeed, it is clear that $\la_{\la_i'} \geq i$, but if
$\la_{\la_i'} > i$, then $\la'_{i+1}=\la'_i$ and so $r_{i+1}=\la_{i+1}-\la'_{i+1} \leq \la_i-\la'_i=r_i$, contradicting choice of $i$.  Thus, $\la_{\la_i'} = i$.
\begin{definition}\label{def:f}
Define a map $f:\Rc^{\ge 1}\to\P$ as follows. Given $\la\in\Rc^{\ge 1}$: 
\begin{enumerate}
    \item Let $i=i(\la)$ be the largest index that minimizes $r_i(\la)$.
    \item Remove a part of size $i$ from $\la$.
    \item Add a part of size $i-1$ to $\la'$. 
    \item Let $f(\la)$ be the resulting partition.
\end{enumerate}

\end{definition}
As an example, in Figure~\ref{fig:fbij_partitions}, $f$ is applied iteratively to a partition until all ranks are positive. 

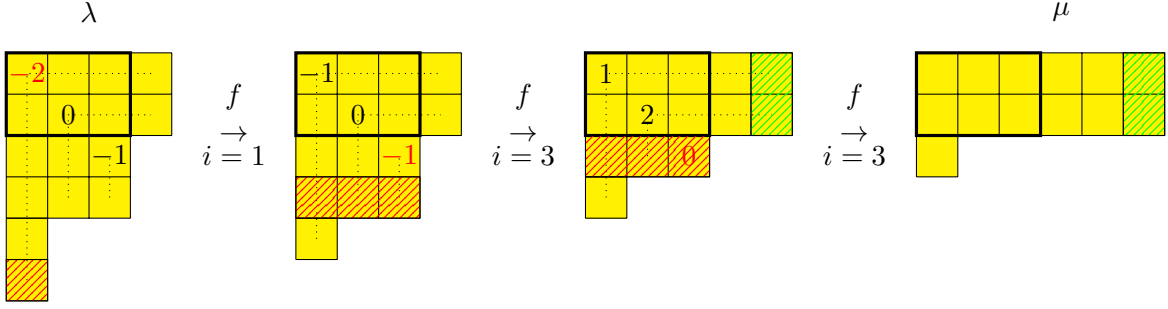
\begin{figure}[h]
\centering
\begin{tikzpicture}[scale=0.55]
    \fill[yellow] (0,0) rectangle (1,2);
    \fill[yellow] (0,2) rectangle (3,4);
    \fill[yellow] (0,4) rectangle (4,6);
    \draw (0,0) grid (1,2);
    \draw (0,2) grid (3,4);
    \draw (0,4) grid (4,6);
    \draw[dotted] (.5,.5)--(.5,5.5)--(3.5,5.5);
    \draw[dotted] (1.5,2.5)--(1.5,4.5)--(3.5,4.5);
    \draw[dotted] (2.5,2.5)--(2.5,3.5);
    \draw[very thick] (0,4) rectangle (3,6);
    \draw[red] (.5,5.5) node {$-2$};
    \draw (1.5,4.5) node {$0$};    
    \draw (2.5,3.5) node {$-1$};
    \draw[pattern=north east lines,pattern color=red] (0,0) rectangle (1,1);
    \draw (2,7) node {$\la$};
    \draw (5.5,4) node [label=$f$] {$\rightarrow$};
    \draw (5.5,4) node [below] {$i=1$};
\begin{scope}[shift={(7,1)}] 
   \fill[yellow] (0,0) rectangle (1,1);
    \fill[yellow] (0,1) rectangle (3,3);
    \fill[yellow] (0,3) rectangle (4,5);
    \draw (0,0) grid (1,1);
    \draw (0,1) grid (3,3);
    \draw (0,3) grid (4,5);
    \draw[dotted] (.5,.5)--(.5,4.5)--(3.5,4.5);
    \draw[dotted] (1.5,1.5)--(1.5,3.5)--(3.5,3.5);
    \draw[dotted] (2.5,1.5)--(2.5,2.5);    
    \draw[very thick] (0,3) rectangle (3,5);
    \draw (.5,4.5) node {$-1$};
    \draw (1.5,3.5) node {$0$};    
    \draw[red] (2.5,2.5) node {$-1$};
    \draw[pattern=north east lines,pattern color=red] (0,1) rectangle (3,2);
    \draw (5.5,3) node [label=$f$] {$\rightarrow$};
    \draw (5.5,3) node [below] {$i=3$};
\end{scope}
\begin{scope}[shift={(14,2)}] 
   \fill[yellow] (0,0) rectangle (1,1);
    \fill[yellow] (0,1) rectangle (3,2);
    \fill[yellow] (0,2) rectangle (5,4);
    \draw (0,0) grid (1,1);
    \draw (0,1) grid (3,2);
    \draw (0,2) grid (5,4);
    \draw[dotted] (.5,.5)--(.5,3.5)--(4.5,3.5);
    \draw[dotted] (1.5,1.5)--(1.5,2.5)--(4.5,2.5);
    \draw[very thick] (0,2) rectangle (3,4);
    \draw (.5,3.5) node {$1$};
    \draw (1.5,2.5) node {$2$};    
    \draw[red] (2.5,1.5) node {$0$};
    \draw[pattern=north east lines,pattern color=red] (0,1) rectangle (3,2);
    \draw[pattern=north east lines,pattern color=green] (4,2) rectangle (5,4);
    \draw (6.5,2) node [label=$f$] {$\rightarrow$};
    \draw (6.5,2) node [below] {$i=3$};
\end{scope}
\begin{scope}[shift={(22,3)}] 
   \fill[yellow] (0,0) rectangle (1,1);
    \fill[yellow] (0,1) rectangle (6,3);
    \draw (0,0) grid (1,1);
    \draw (0,1) grid (6,3);
    \draw[very thick] (0,1) rectangle (3,3);
    \draw[pattern=north east lines,pattern color=green] (5,1) rectangle (6,3);
    \draw (3.5,4) node {$\mu$};
\end{scope}
\end{tikzpicture}
  \caption{Iterating $f$ from Definition~\ref{def:f}:  The last occurrence of the smallest rank $r_i$ is colored in red, the part of size $i$ to be removed shaded in red, and the added part of size $i-1$ in the conjugate is shaded in green. The thicker rectangle is the Durfee rectangle, which is preserved by $f$. As long as the smallest rank in negative, the Durfee square is preserved by $f$ as well.}\label{fig:fbij_partitions}
\end{figure}

\begin{lemma}
\label{f_lemma}
Let $\la\in\Rc^{\ge1}$ and let $\mu = f(\la)$.  Then
\begin{enumerate}[(a)]
\item $|\mu|=|\la|-1$;
\item 
$\mu'_1=\la'_1-1$; 
\item $\mu_1=\la_1+1$ unless $i(\la)=1$, in which case $\mu_1=\la_1$;
\item $\tau(\mu) > \tau(\la)$;  if $\tau(\la) < 0$ then $\tau(\mu) = \tau(\la)+1$;
\item $i(\mu) \geq i(\la)$;
\item $\dr(\mu) = \dr(\la)$; if $\tau(\la) < 0$ then $d(\mu) = d(\la)$.
\end{enumerate}

\end{lemma}

\begin{proof}
Let $i=i(\la)$. 
When applying the map $f$ to $\la$, each $\la_j$ for $1 \leq j < i$ increases by 1, and each $\la'_j$ for $1 \leq j \leq i$  decreases by 1, proving (a)--(c).
The effect of $f$ on the ranks is that each $r_j=\la_j-\la'_j$ for $1 \leq j < i$ increases by 2.
Additionally, $r_i=\la_i-\la'_i$ increases by 1 and the Durfee square and rectangle are both preserved, 
unless $i=d=\la_d > \la_{d+1}$ (which necessarily implies $\tau(\la)=0$). In this special case, removing a part of size $d$ from $\la$ decreases the size of the Durfee square, but adding a part of size $d-1$ to $\la'$ still preserves the Durfee rectangle.  This proves (d)--(f).
\end{proof}

\begin{lemma}\label{lem:f}
Let $m,n,\ell\ge0$ be such that $n+\ell\ge m$.
\begin{enumerate}[(a)]
\item For $\ell\ge1$, the map $f$ is a bijection 
$$f:\Rc_{m,n}^{\ge 1-\ell}(N)\to \Rc_{m-1,n+1}^{\ge 2-\ell}(N-1)$$
that preserves the Durfee square and the Durfee rectangle.
\item For $\ell=0$, the map $f$ is a bijection 
$$f:\Rc_{m,n}^{\ge 1}(N)\to \P_{m-1,n+1}(N-1)$$
that preserves the Durfee rectangle.
\end{enumerate}
\end{lemma}

\begin{proof}
By Lemma \ref{f_lemma}, in both claims (a) and (b), clearly $f$ sends a partition in the domain to one in the specified range, preserving the Durfee square and/or rectangle as claimed.  To show that $f$ is onto, we will define an inverse $g$.

First note that for a partition $\la$ with $\tau(\la) > 0$, $\la'$ must contain a part of size $d(\la)$. 
If $\tau(\la) \leq 1$ and $j$ is the smallest index such that $r_j=\tau(\la)$ and $j>1$, then $\la'$ contains a part of size $j-1$.

Define $g:\mathcal{P}\rightarrow \mathcal{P}$  for $\la \in \mathcal{P}$ as follows:
\begin{enumerate}
    \item Let $\tau$ be the smallest rank of $\la$.
    \item If $\tau>1$, let $j=d(\la)+1$; otherwise, let $j$ be the smallest index that minimizes $r_j(\la)$.
    \item Remove a part of size $j-1$ from $\la'$.
    \item Add a part of size $j$ to $\la$. 
    \item Let $g(\la)$ be the resulting partition.
\end{enumerate}

Continuing with the proof of Theorem \ref{lem:f}(a), if $\ell \geq 1$ and $\la \in \Rc_{m-1,n+1}^{\ge 2-\ell}(N-1)$, then $g$ decreases $\la_k$ by 1 for
$1 \leq k < j$ and increases $\la_k'$ by 1 for $1 \leq k \leq j$. Thus $r_j$ decreases by 1 and $r_k$ decreases by 2 for $1 \leq k < j$.
So 
$g(\la) \in \Rc^{\ge 2-\ell}(N)$ and  $f(g(\la))=\lambda$.  The only concern is whether $g(\la) \in \mathcal{P}_{m,n}$.  The map $g$ increases the number of parts of $\la$ by 1, but if $j(\la)=1$, $g$ does not decrease $\la_1$.  So $g(\la)$ will not belong to $\P_{m,n}$ if $j(\la)=1$ and $\la_1=n+1$.
But this means that $\tau(\la)=\la_1-\la'_1 = n+1-\la_1'$.  Combining with the fact that, since $\la \in \Rc_{m-1,n+1}^{\ge 2-\ell}(N-1)$, we must have
$\tau(\la) \leq 1-\ell$ and $\la_1' \leq m-1$, this gives
$$1-\ell \geq \tau(\lambda) \geq n-m+2,$$
contradicting the assumption that $n+\ell \geq m$.

To complete the proof of Lemma \ref{lem:f}(b), we need only  consider those $\la \in \mathcal{P}_{m-1,n+1}$ with $\tau(\la) > 0$ (those with $\tau(\la) \leq 0$ are covered by (a)) and show that
$g(\la) \in \Rc_{m,n}^{\ge 1}(N)$ and $f(g(\la))=\la.$

If $\tau(\la) \geq 2$, the map $g$ removes a part of size $d$ from $\la'$ and adds a part of size $d+1$ to $\la$, which increases the side of the Durfee square.  So $\mu =g(\la)$ satisfies that $d(\mu) = d(\la) + 1$, $r_{d(\mu)}(\mu) = 0$, and $d(\mu) = \mu_{d(\mu)} >\mu_{d(\mu)+1}$.  So, in this case,
$g(\la) \in \Rc_{m,n}^{\ge 1}(N)$ and $f(g(\la))=\la.$

In the remaining case, $\tau(\la) =1$.  As in (a), if $j=j(\la)$,
$g$ decreases $r_j$  by 1 and decreases $r_k$  by 2 for $1 \leq k < j$.
So $g(\la) \in \Rc^{\ge 1}(N)$ and  $f(g(\la))=u$.  As in case (a), the condition $n \geq m$ guarantees that $g(\la) \in \mathcal{P}_{m,n}$ even if $j(\la)=1$.
\end{proof}

Iterating $f$, we obtain the following bijections.

\begin{theorem}\label{thm:bij}
Let $m,n,\ell\ge0$ such that $n+\ell\ge m$. 
\begin{enumerate}[(a)]
\item The map $\bij:=f^{\ell+1}$ is a bijection
$$
\bij:\Rc_{m,n}^{\ge 1-\ell}(N)\to \P_{m-\ell-1,n+\ell+1}(N-\ell-1)    
$$
that preserves the Durfee rectangle.

\item The map $f^\ell$ is a bijection
$$
f^\ell:\Rc_{m.n}^{\ge 1-\ell}(N)\to   \Rc_{m-\ell,n+\ell}^{\ge 1}(N-\ell)
$$
that preserves the Durfee square and the Durfee rectangle.
\end{enumerate}
\end{theorem}

\begin{example}
In Figure \ref{fig:fbij_partitions}, if we let $m=6$, $n=4$, $\ell=2$, and view $\la$ as an element of
$\Rc_{6,4}^{\ge -1}(16)$, then
$f(\la) \in  \Rc_{5,5}^{\ge 0}(15)$,
$f^2(\la) \in  \Rc_{4,6}^{\ge 1}(14)$, and
$\Theta(\la) =f^3(\la) \in  \P_{3,7}(13$). Here, the bijection $\Theta$ preserves the Durfee rectangle, but not the Durfee square, whereas the bijection $f^2$ preserves both.
\end{example}

We can now use Theorem \ref{thm:bij} to prove Theorems~\ref{thm:central_drect} and~\ref{thm:central-simplified}.

\begin{proof}[Proof of Theorem~\ref{thm:central_drect}]
We express the sum over $\R_{m,n}^{\ge 1-\ell}$ as
\begin{equation}\sum_{\la \in \R_{m,n}^{\ge 1-\ell}}t^{\dr(\la)}q^{|\la|}=
\sum_{\la \in \P_{m,n}} t^{\dr(\la)}q^{|\la|} - 
\sum_{\la \in \Rc_{m,n}^{\ge 1-\ell}}t^{\dr(\la)}q^{|\la|}.
\label{eq:drcount}
\end{equation}

A partition $\la \in \P_{m,n}$ with $\dr(\la)=i$ can be viewed as an $i \times (i+1)$ rectangle with a partition in $\P_{i,n-i-1}$ to its right and the conjugate of a partition in $\P_{i+1,m-i}$ below. So

\begin{equation}
\sum_{\la \in \P_{m,n}} t^{\dr(\la)}q^{|\la|} =
\sum_{i \geq 0}t^iq^{i(i+1)}\qbin{n-1}{i} \qbin{m+1}{i+1}.
\label{eq:drPmn}
\end{equation}

By Theorem \ref{thm:bij}(a) and equation \eqref{eq:drPmn},
\begin{equation}
\sum_{\la \in \Rc_{m,n}^{\ge 1-\ell}}t^{\dr(\la)}q^{|\la|} = 
\sum_{\la \in \P_{m-\ell-1,n+\ell+1}} t^{\dr(\la)}q^{\ell +1+|\la|} =
q^{\ell+1}\sum_{i \geq 0}t^iq^{i(i+1)}\qbin{n+\ell}{i} \qbin{m-\ell}{i+1}.
\label{eq:dr_thm_a}
\end{equation}
Combining equations  \eqref{eq:drcount},  \eqref{eq:drPmn}, and \eqref{eq:dr_thm_a}, we obtained the stated formula.
\end{proof}

\begin{proof}[Proof of Theorem~\ref{thm:central-simplified}]
We now count $\R_{m,n}^{\ge 1-\ell}$  according to the side of the Durfee square as
\begin{equation}\sum_{\la \in \R_{m,n}^{\ge 1-\ell}}t^{d(\la)}q^{|\la|}=
\sum_{\la \in \P_{m,n}} t^{d(\la)}q^{|\la|} - 
\sum_{\la \in \Rc_{m,n}^{\ge 1-\ell}}t^{d(\la)}q^{|\la|}.
\label{eq:dcount}
\end{equation}
By Theorem \ref{thm:bij}(b),
\begin{equation}
\sum_{\la \in \Rc_{m,n}^{\ge 1-\ell}}t^{d(\la)}q^{|\la|}=
q^{\ell}\sum_{\la \in \Rc_{m-\ell,n+\ell}^{\ge 1}}t^{d(\la)}q^{|\la|}.
\label{eq:d_thm_b}
\end{equation}
By conjugation, and then by equations~\eqref{eq:phib} and~\eqref{eq:RV},
$$\sum_{\la \in \R_{m-\ell,n+\ell}^{\ge 1}}t^{d(\la)}q^{|\la|} =
\sum_{\la \in \R_{n+\ell,m-\ell}^{\le -1}}t^{d(\la)}q^{|\la|} =
\sum_{P \in \V_{n+\ell,m-\ell}^{\geq 0}} t^{\des(P)}q^{\maj(P)}=
\sum_{P \in \A_{n+\ell,m-\ell}^{0}} t^{\des(P)}q^{\maj(P)}.
$$
Note that $\V_{n+\ell,m-\ell}^{\geq 0}=\A_{n+\ell,m-\ell}^{0}$ since the conditions $\ell \geq 0$ and $n+\ell \geq m$ guarantee that $(n+\ell)-(m-\ell) \geq 0$.
Taking complements above, the sequence of equalities  still holds, so now apply equation~\eqref{Ac0} to get
$$\sum_{\la \in \Rc_{m-\ell,n+\ell}^{\ge 1}}t^{d(\la)}q^{|\la|} =
\sum_{P \in \Ac_{n+\ell,m-\ell}^{0}} t^{\des(P)}q^{\maj(P)} =
\sum_{i \ge0}t^iq^{i^2}\qbin{n+\ell-1}{i-1}\qbin{m-\ell+1}{i+1}.$$
Combining this with equations \eqref{eq:dcount} and \eqref{eq:d_thm_b}, and using the right equality in Lemma~\ref{lem:qbin2}, we obtain
$$\sum_{\la \in \R_{m,n}^{\ge 1-\ell}}t^{d(\la)}q^{|\la|}=
\sum_{i \geq 0}t^iq^{i^2}\qbin{n}{i} \qbin{m}{i} -
q^{\ell}\sum_{i \ge0}t^iq^{i^2}\qbin{n+\ell-1}{i-1}\qbin{m-\ell+1}{i+1}.\qedhere
$$
\end{proof}

\subsection{Removing the bounding box}

Disregarding the side of the Durfee square, the bijection $\bij$ from Theorem~\ref{thm:bij} gives the following formula, which corresponds to setting $t=1$ in Theorems~\ref{thm:central_drect} or~\ref{thm:central-simplified}.

\begin{theorem}\label{thm:central_t1}
Let $m,n,\ell\ge0$ such that $n+\ell\ge m$.  Then
$$
\sum_{\la\in \R^{\ge 1-\ell}_{m,n}} q^{|\la|}
=\qbin{m+n}{m}-q^{\ell+1}\qbin{m+n}{m-\ell-1}.
$$
\end{theorem}

\begin{proof}
By Theorem~\ref{thm:bij}(a), 
$$\sum_{\la\in \Rc_{m,n}^{\ge 1-\ell}} q^{|\la|}=q^{\ell+1}\sum_{\la\in\P_{m-\ell-1,n+\ell+1}}q^{|\la|}=q^{\ell+1}\qbin{m+n}{m-\ell-1}.$$
\end{proof}

Combining Theorem~\ref{thm:central_t1} when $m=n$ and Corollary~\ref{cor:lopsided} with $t=1$, we get the following. Note that for $\ell\in\{0,1\}$, both cases apply.

\begin{corollary}\label{cor:combined_t1}
For all $\ell\in\Z$,
$$
\sum_{\la\in \R^{\ge 1-\ell}_n} q^{|\la|}
=\begin{cases}
C_{n+\ell}(q,q^{-\ell}) & \text{if $\ell\le1$}, \\
\qbin{2n}{n}-q^{\ell+1}\qbin{2n}{n-\ell-1} & \text{if $\ell\ge0$}.
\end{cases}
$$
\end{corollary}

In Theorems \ref{thm:central_drect}, \ref{thm:central-simplified} and \ref{thm:central_t1}, partitions were required to be inside a rectangle. Next we obtain corollaries of these theorems by taking the limit as the side lengths of the rectangle go to infinity.

Letting $n\to\infty$ in Theorem~\ref{thm:central_t1} gives the following. We add the condition $m>\ell$ because, when it does not hold, the right summand in Theorem~\ref{thm:central_t1} is $0$.

\begin{corollary}
For $m>\ell\ge0$,
$$
\sum_{\la\in \R^{\ge 1-\ell}_{m,\infty}} q^{|\la|}=
\prod_{i=1}^{m}\frac{1}{1-q^i}-q^{\ell+1}\prod_{i=1}^{m-\ell-1}\frac{1}{1-q^i}.
$$
\end{corollary}

Letting now $m\to\infty$, we recover the following result, which is equivalent to Corollary~\ref{cor:AB}.

\begin{corollary}
For $\ell\ge0$,
$$
    \sum_{\la\in \R^{\ge 1-\ell}} q^{|\la|}=
    \prod_{\substack{i\ge1\\ i\neq \ell+1}}\frac{1}{1-q^i}.
$$
\end{corollary}

Letting first $n\to\infty$ and then $m\to\infty$ in Theorem~\ref{thm:central_drect}, we obtain the following.

\begin{corollary}
For $m,\ell\ge0$,
$$
    \sum_{\la\in \R_{m,\infty}^{\ge 1-\ell}} t^{\dr(\la)}q^{|\la|}=
    \sum_{i\ge0}\frac{t^i q^{i(i+1)}}{(q)_i}\left(\qbin{m+1}{i+1}-q^{\ell+1}\qbin{m-\ell}{i+1}\right).
$$
\end{corollary}

\begin{corollary}
For $\ell\ge0$,
$$
    \sum_{\la\in \R^{\ge 1-\ell}} t^{\dr(\la)}q^{|\la|}=
    (1-q^{\ell+1})\sum_{\mu\in \P} t^{\dr(\mu)}q^{|\mu|}=
    (1-q^{\ell+1})\sum_{i\ge0}\frac{t^i q^{i(i+1)}}{(q)_i(q)_{i+1}}.
$$
\end{corollary}

Similarly, letting first $n\to\infty$ and then $m\to\infty$ in Theorem~\ref{thm:central-simplified}, we obtain the following.

\begin{corollary}\label{cor:main-limit}
For $m,\ell\ge0$,
$$\sum_{\la\in \R_{m,\infty}^{\ge 1-\ell}} t^{d(\la)}q^{|\la|}
    =\sum_{i\ge1} \frac{t^i q^{i^2}}{(q)_i} \left( \qbin{m}{i}-  {q^\ell(1-q^i)}\qbin{m-\ell+1}{i+1}\right).$$
\end{corollary}

\begin{corollary}\label{cor:mainlimit}
For $\ell\ge0$,
$$
    \sum_{\la\in \R^{\ge 1-\ell}} t^{d(\la)}q^{|\la|}
    =\sum_{i\ge0} t^i q^{i^2}\left(\frac{1}{(q)_i^2}-\frac{q^{\ell}}{(q)_{i-1}(q)_{i+1}}\right)
    =\sum_{i\ge0} t^i q^{i^2}\frac{(1-q^{i+1})-q^{\ell}(1-q^i)}{(q)_{i}(q)_{i+1}}.
$$
\end{corollary}

\section{The bijection $\bij$ in terms of lattice paths}
\label{sec:theta}

In~\cite{GK}, Greene and Kleitman define a mapping in order to construct a symmetric chain decomposition of the Boolean lattice $B_n$, consisting of the subsets of $\{1, 2, \ldots, n\}$ ordered by inclusion.  Such a subset can be represented by a word $w_1w_2 \ldots w_n$ over $\{1,2\}$ where $w_i = 2$ if $i$ is in the subset, and $w_i=1$ otherwise.

To assist in defining this mapping, which we denote by $\gamma$, let $w=w_1w_2 \ldots w_n\in B_n$, regard each 1 as a left parenthesis and each 2 as a right parenthesis, and match parentheses in $w$ in the usual way as follows. Every 1 followed immediately by a 2 in $w$ form a matched pair, remove matched pairs from $w$ and repeat. The procedure stops when all remaining 1's are to the right of all remaining 2's so that no further matching occurs.  For example, the word $22112122111$ has matched pairs at positions $(4,5)$, $(6,7)$, and $(3,8)$, with unmatched 2's at positions 1,2 and unmatched 1's at positions $9, 10, 11$.

For $w \in B_n$ with at least one unmatched 2, the mapping $\gamma$ changes the rightmost unmatched 2 to a 1.  Note that $\gamma(w)$ has the same set of matched pairs as $w$.  An inverse can be defined on any $w \in B_n$ with at least one unmatched 1:  $\gamma^{-1}$ changes the leftmost unmatched 1 to a 2.  With this construction, Greene and Kleitman showed that the chains formed by starting from a word $w \in B_n$ with no  unmatched 1 and iterating $\gamma$ until there is no unmatched 2 give a symmetric chain decomposition of $B_n$.

In this section we show that, when the partition bijections from Section~\ref{sec:ranks_infinite_interval-central} are translated into lattice path bijections via Foata's correspondence, they have a simple description in terms of $\gamma$.

\subsection{The Greene--Kleitman mapping for lattice paths}\label{sec:GK}

Let us first interpret $\gamma$ in terms of lattice paths.
To be consistent with the parameters for partitions, in this section we use $n$ and $m$, in this order, to denote the number of up and down steps of a path. 

As in Section~\ref{sec:lattice_paths}, we view words in $\W_{n,m}$, consisting of $n$ 1's and $m$ 2's, as lattice paths in $\G_{n,m}$, which start at the origin and have $n$ steps $U=(1,1)$ and $m$ steps $D=(1,-1)$.  For $P \in \G_{n,m}$, let $\min(P)$ denote the minimum $y$-coordinate of any point of $P$. A point on $y=\min(P)$ is called a {\em minimum} of $P$.

Given $P=p_1p_2 \ldots p_{n+m} \in \G_{n,m}$, in order to define $\gamma(P)$, first match the $U$ and $D$ steps the same way as the 1's and 2's in the corresponding word in $\W_{n,m}$.  Any unmatched $D$s must come before any unmatched $U$s.

Removing a matched $UD$ pair from $P$ does not change $\min(P)$.
So if $p_{i_j}$ is the $j$-th unmatched $D$ of $P$, then this step goes from $(i_j-1,-j+1)$ to $(i_j,-j)$.  And since the portion of $P$ between $p_{i_{j-1}}$ and $p_{i_j}$ is matched, it does not go below $y=-j+1$.  It follows that the rightmost unmatched $D$ of $P$ (if any) is the one ending at the leftmost minimum of $P$.  Similarly, the leftmost unmatched $U$ of $P$ (if any) is the one starting at the rightmost minimum of $P$. Also, the number of unmatched  $D$ steps in $P$ is $-\min(P)$ and the number of unmatched $U$ steps is $n-m-\min(P)$.  Summarizing, we can give an equivalent description of $\gamma$ and $\gamma^{-1}$ in terms of lattice paths as follows.

\begin{definition}\label{def:gamma}
Let $m,n\ge0$ and $P \in \G_{n,m}$. 

If $\min(P) < 0$,
consider the leftmost minimum of $P$. Turn the $D$ step ending at this point into a $U$ step. Let $\gamma(P)$ be the resulting path. 

If $\min(P)<n-m$,
consider the rightmost minimum of $P$. Turn the $U$ step starting at this point into a $D$ step. Let $\gamma^{-1}(P)$ be the resulting path. 
\end{definition}

\begin{example}
In Figure~\ref{fig:GK-example}, the path $P=DDUUDUDDUUU \in \G_{6,5}$ has $2=-\min(P)$ unmatched D steps (at positions 1,2) and $3=6-5-\min(P)$ unmatched U steps (at positions $9, 10, 11$).  The path on the right in the figure  is $\gamma(P) = DUUUDUDDUUU$.
\end{example}

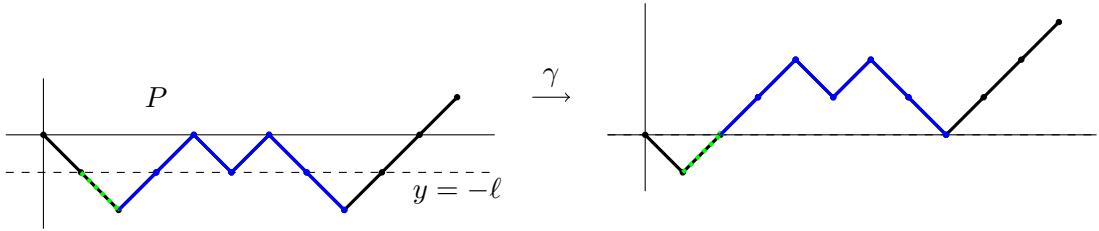
\begin{figure}[h]
    \centering
    \begin{tikzpicture}[scale=.5]
       \draw (-1,0)--(12,0);
       \draw (0,-2.5)--(0,1.5);
       \draw[dashed] (-1,-1)--(12,-1);
     \draw[very thick](0,0) circle(1.2pt)\down\down\up\up\down\up\down\down\up\up\up; 
     \draw[green,ultra thick,dotted] (1,-1) \down;
     \draw[blue,very thick] (2,-2) \up\up\down\up\down\down;
     \draw (3,1) node {$P$};
      \draw[->] (13,1)-- node[above] {$\gamma$} (14,1);
      \draw (11,-1.5) node {$y=-\ell$};
\begin{scope}[shift={(16,0)}]
       \draw (-1,0)--(12,0);
       \draw (0,-1.5)--(0,3.5);
       \draw[dashed] (-1,0)--(12,0);
     \draw[very thick](0,0) circle(1.2pt)\down\up\up\up\down\up\down\down\up\up\up; 
     \draw[green,ultra thick,dotted] (1,-1) \up;
     \draw[blue,very thick] (2,0) \up\up\down\up\down\down;
\end{scope}
    \end{tikzpicture}
    \caption{An example of the map $\gamma:\B^{-1}_{6,5}\to\B^{0}_{7,4}$ from Definition~\ref{def:gamma}. The matched steps are colored in blue, and the rightmost unmatched $D$ step in $P$, which becomes a $U$ step in $\gamma(P)$, is dotted in green. }
    \label{fig:GK-example}
\end{figure}

We now restrict $\gamma$ to  $\B^{-\ell}_{n,m}$,
the set of all  $P \in \G_{n,m}$
with $\min(P) < -\ell$.
The following lemma generalizes the lemma in~\cite[p.\ 255]{FH}. Note that the condition $n+\ell\ge m$ simply states that the right endpoint of paths in $\B^{-\ell}_{n,m}$ is on or above the line $y=-\ell$.

\begin{lemma}\label{lem:GK}
Let $m,n,\ell\ge0$. 
If $n+\ell\ge m-1$, then the map $\gamma:\B^{-\ell}_{n,m}\to\B^{-\ell+1}_{n+1,m-1}$ from Definition~\ref{def:gamma} is a bijection.
Additionally, if $n+\ell\ge m$, then for all $P\in\B^{-\ell}_{n,m}$ we have
\begin{align*} \maj(\gamma(P))&= \maj(P)-1;\\
\des(\gamma(P)) & = \begin{cases} \des(P)-1, & \text{if $\ell=0$, and $P$ begins with a $D$ and $\min(P)\ge-1$};\\
\des(P), & \text{otherwise.} \end{cases} 
\end{align*}
\end{lemma}

\begin{proof}
For $P\in\B^{-\ell}_{n,m}$, it is clear that $\gamma(P)\in\B^{-\ell+1}_{n+1,m-1}$, since 
$\min(\gamma(P)) = \min(P)+1$.  
To see that $\gamma$ is a bijection, note that the condition $n+\ell \geq m-1$ guarantees that any $P'\in \B^{-\ell+1}_{n+1,m-1}$ has an unmatched $U$, since
$\min(P') \leq -\ell$ and the number of unmatched $U$ steps is 
$$n+1-(m-1) -\min(P') \geq n-m+2+\ell \geq 1.$$
Thus, we can uniquely recover $\gamma^{-1}(P')$ by considering the rightmost minimum of $P'$ and turning the $U$ step starting at this point into a $D$ step.

To show that $\maj(\gamma(P))=\maj(P)-1$, first we note that the condition $n+\ell\ge m$ guarantees that the right endpoint of $P\in\B^{-\ell}_{n,m}$ is at height $n-m\ge-\ell > \min(P)$. Thus, the minima of $P$ are valleys.
By construction, $\gamma$ preserves the positions ($x$-coordinates) of all the valleys of $P$, except for the valley at the leftmost minimum, which is moved one position to the left if it occurs at $x>1$, or which disappears if $x=1$. In both cases, $\gamma$ decreases the major index by one.

Finally, for the statement about $\des$, we observe that this statistic is preserved by $\gamma$ unless the step changed by $\gamma$ is the first step of $P$, in which case the number of valleys decreases by one. This happens precisely when $P$ starts with a $D$ and the point at the end of this step is a minimum.
\end{proof}

Repeated applications of Lemma~\ref{lem:GK} give the following result.

\begin{lemma}\label{lem:GKrepeat}
Let $m,n,\ell\ge0$ such that $n+\ell\ge m$. 
\begin{enumerate}[(a)]
    \item The map $\gamma^{\ell+1}:\B^{-\ell}_{n,m}\to\G_{n+\ell+1,m-\ell-1}$ is a bijection such that, for all $P\in\B^{-\ell}_{n,m}$,
\begin{align}\nonumber \maj(\gamma^{\ell+1}(P))&= \maj(P)-\ell-1;\\
\label{eq:desGK}
    \des(\gamma^{\ell+1}(P))&=\begin{cases} \des(P)-1, & \text {if $P$ begins with a $D$ and $\min(P)\ge-\ell-1$};\\
    \des(P), & \text{otherwise}.
\end{cases}
\end{align}
\item The map $\gamma^{\ell}:\B^{-\ell}_{n,m}\to\B^0_{n+\ell,m-\ell}$ is a bijection such that, for all $P\in\B^{-\ell}_{n,m}$,
$$\maj(\gamma^{\ell}(P))= \maj(P)-\ell \qquad \text{and} \qquad \des(\gamma^{\ell}(P))=\des(P).$$
\end{enumerate}
\end{lemma}

\begin{proof}
Applying Lemma~\ref{lem:GK} $\ell+1$ times, we obtain a sequence of bijections
$$\B^{-\ell}_{n,m}\to\B^{-\ell+1}_{n+1,m-1}\to\dots\to\B^{0}_{n+\ell,m-\ell}\to\B^{1}_{n+\ell+1,m-\ell-1}=\G_{n+\ell+1,m-\ell-1}.$$ 
Since each application of $\gamma$ decreases the major index by one, the statements about $\maj$ follow.

Similarly, it follows from Lemma~\ref{lem:GK} that $\gamma^\ell$ preserves the number of descents. 
To see how this statistic behaves under $\gamma^{\ell+1}$, note that this map can be described in one step as follows. Given $P\in\B^{-\ell}_{n,m}$, match $U$ and $D$ steps as if they were opening and closing parentheses, and then turn the $\ell+1$ rightmost unmatched $D$ steps into a $U$ steps. The resulting path is $\gamma^{\ell+1}(P)$.

From this description, it is clear that the number of valleys of the path is preserved by $\gamma^{\ell+1}$ unless this map changes the first step of $P$ from a $D$ to a $U$, in which case the number of valleys decreases by one. This happens when $P$ starts with a $D$ and this $D$ is one of the $\ell+1$ rightmost unmatched $D$ steps; equivalently, $P$ starts with a $D$ and has minimum at height $-\ell-1$.
\end{proof}

\subsection{Relating $f$ and $\gamma$}

For a set of paths $\mathcal{H}$, it will be convenient to use the notation $\mathcal{H}(N)=\{P\in\mathcal{H}:\maj(P)=N\}$.

\begin{theorem}\label{thm:f=gamma}
Let $m,n,\ell\ge0$ be such that $n+\ell\ge m$.
The bijection $f$ from Definition~\ref{def:f} and the bijection $\gamma$ from Definition~\ref{def:gamma} are related by Foata's bijection $\phi^{-1}$ in the following way: $f(\la)=\phi(\gamma(\phi^{-1}(\la')))$.  This is illustrated in the following commutative diagram, where $\conj$ denotes conjugation of partitions: 
$$
\begin{array}{ccccc}
\Rc_{m,n}^{\ge 1-\ell}(N)&\stackrel{\conj}{\longleftrightarrow}& \Rc_{n,m}^{\le \ell-1}(N)&\stackrel{\phi^{-1}}{\longrightarrow}&
\B_{n,m}^{-\ell}(N)\medskip\\
\downarrow\footnotesize{f}
&&&&\downarrow\footnotesize{\gamma}\\
\Rc_{m-1,n+1}^{\ge 2-\ell}(N-1)&\stackrel{\conj}{\longleftrightarrow}& \Rc_{n+1,m-1}^{\le \ell-2}(N-1)&\stackrel{\phi^{-1}}{\longrightarrow}&\B_{n+1,m-1}^{-\ell+1}(N-1)
\end{array}
$$
\end{theorem}

\begin{proof}
Let us first check that all the maps in the diagram are bijections. This is clear for the conjugation map on partitions, and for the maps $f$ and $\gamma$ by Lemmas~\ref{lem:f} and~\ref{lem:GK}. 
By equation~\eqref{eq:phib}, the map $\phi^{-1}$ restricts to a bijection between 
$\R_{n,m}^{\le \ell-1}$ and $\V^{\ge-\ell}_{n,m}$.
The condition $n+\ell\ge m$ guarantees that paths in the latter set end at height $n-m\ge -\ell$, and thus not only their valleys but also the entire paths lie on or above the line $y=-\ell$. It follows that $\V^{\ge-\ell}_{n,m}=\A^{-\ell}_{n,m}$, and so $\phi^{-1}$ restricts to a bijection between $\Rc_{n,m}^{\le \ell-1}$ and 
$\B^{-\ell}_{n,m}$. Similarly, $n+\ell\geq m$ guarantees that $(n+1)+(\ell-1) \geq m-1$, so
$\V^{\ge-\ell+1}_{n+1,m-1}=\A^{-\ell+1}_{n+1,m-1}$, and  $\phi^{-1}$ restricts to a bijection between $\Rc_{n+1,m-1}^{\le \ell-2}$ and 
$\B^{-\ell+1}_{n+1,m-1}$.
Finally, by Lemma~\ref{lem:Foata}, $\phi^{-1}$ sends the area of the partition to the major index of the path.

Next we show that $f(\la)=\phi(\gamma(\phi^{-1}(\la'))')$.
 Let $\la \in \Rc_{m,n}^{\ge 1-\ell}(N)$ and $d=d(\lambda)$.  
 Consider the hook decomposition of its conjugate $\la'$, given Definition~\ref{def:hook_decomp} but with $\la'$ playing the role of $\la$.
 (See Figure \ref{fig:comparison}, where the $a_i$ and $b_i$ of $\la'$ are shown in blue and red, respectively.)
 
\begin{figure}[h!]
\centering
    \begin{tikzpicture}[scale=0.55]
\begin{scope}[shift={(0,0)}] 
    \fill[yellow] (0,0)--(1,0)--(1,2)--(3,2)--(3,4)--(4,4)--(4,6)--(0,6)--(0,0);
    \fill[orange!40] (0,4) rectangle (3,6);
      \draw[gray,thin] (0,0) grid (5,6);
      \draw[very thick] (0,0) \east\north\north\east\east\north\north\east\north\north\east;
     \draw[dotted] (.5,.5)--(.5,5.5)--(3.5,5.5);
    \draw[dotted] (1.5,2.5)--(1.5,4.5)--(3.5,4.5);
    \draw[dotted] (2.5,2.5)--(2.5,3.5);
    \draw[red] (.5,5.5) node {$-2$};
    \draw (1.5,4.5) node {$0$};    
    \draw (2.5,3.5) node {$-1$};
    \fill[pattern=north east lines,pattern color=red] (0,0) rectangle (1,1);  
      \draw (6,3) node [label=$\conj$] {$\longleftrightarrow$};
      \draw (2.5,6.7) node {$\la$};
    \draw[->] (2.5,-.5)-- node[right] {$f$} (2.5,-1.5);
    \draw (1,-1) node {$i=1$};
\end{scope}
\begin{scope}[shift={(7,.5)}] 
    \fill[blue!30] (0,4) rectangle (6,5);
    \fill[blue!30] (1,3) rectangle (4,4);
    \fill[blue!30] (2,2) rectangle (4,3);
    \fill[red!30] (0,4) rectangle (1,1);
    \fill[red!30] (1,3) rectangle (2,1);
      \draw[gray,thin] (0,0) grid (6,5);
      \draw[very thick] (0,0) \north\east\east\north\east\east\north\north\east\east\north;
      \draw[red] (.5,1) node [below] {$3$};
      \draw[red] (1.5,1) node [below] {$2$};
      \draw[red] (2.5,2) node [below] {$0$};
      \draw[blue] (4,2.5) node[right] {$2$};
      \draw[blue] (4,3.5) node[right] {$3$};
      \draw[blue] (6,4.5) node[right] {$6$};
      \draw (3,5.7) node {$\la'$};
      \draw (7,2.5) node [label=$\phi^{-1}$] {$\longrightarrow$};
\end{scope}
\begin{scope}[shift={(0,-7)}] 
 \fill[yellow] (0,0)--(1,0)--(1,1)--(3,1)--(3,3)--(4,3)--(4,5)--(0,5)--(0,0);
 \fill[orange!40] (0,3) rectangle (3,5);
     \draw[gray,thin] (0,0) grid (6,5);
   \draw[very thick] (0,0) \east\north\east\east\north\north\east\north\north\east\east;
    \draw[dotted] (.5,.5)--(.5,4.5)--(3.5,4.5);
    \draw[dotted] (1.5,1.5)--(1.5,3.5)--(3.5,3.5);
    \draw[dotted] (2.5,1.5)--(2.5,2.5);    
    \draw (.5,4.5) node {$-1$};
    \draw (1.5,3.5) node {$0$};    
    \draw[red] (2.5,2.5) node {$-1$};
    \fill[pattern=north east lines,pattern color=red] (0,1) rectangle (3,2);
     \draw (8,2.5) node [label=$\conj$] {$\longleftrightarrow$};
     \draw (10.5,2.5) node {$\dots$};
      \draw (13,2.5) node [label=$\phi^{-1}$] {$\longrightarrow$};
   \draw[->] (2.5,-.5)-- node[right] {$f$} (2.5,-1.5);
    \draw (1,-1) node {$i=3$};
\end{scope}
\begin{scope}[shift={(0,-13)}] 
 \fill[yellow] (0,0)--(1,0)--(1,1)--(3,1)--(3,2)--(5,2)--(5,4)--(0,4)--(0,0);
 \fill[orange!40] (0,2) rectangle (3,4);
 \draw[gray,thin] (0,0) grid (7,4);
    \draw[dotted] (.5,.5)--(.5,3.5)--(4.5,3.5);
    \draw[dotted] (1.5,1.5)--(1.5,2.5)--(4.5,2.5);
   \draw[very thick] (0,0) \east\north\east\east\north\east\east\north\north\east\east;
    \draw (.5,3.5) node {$1$};
    \draw (1.5,2.5) node {$2$};    
    \draw[red] (2.5,1.5) node {$0$};
    \fill[pattern=north east lines,pattern color=red] (0,1) rectangle (3,2);
     \draw (9,2.5) node [label=$\conj$] {$\longleftrightarrow$};
     \draw (11,2.5) node {$\dots$};
      \draw (13,2.5) node [label=$\phi^{-1}$] {$\longrightarrow$};
    \fill[pattern=north east lines,pattern color=green] (4,2) rectangle (5,4);
  \draw[->] (2.5,-.5)-- node[right] {$f$} (2.5,-1.5);
    \draw (1,-1) node {$i=3$};
\end{scope}
\begin{scope}[shift={(0,-18)}] 
   \fill[yellow] (0,0)--(1,0)--(1,1)--(6,1)--(6,3)--(0,3)--(0,0);
    \fill[orange!40] (0,1) rectangle (3,3);
  \draw[gray,thin] (0,0) grid (8,3);
      \draw[very thick] (0,0) \east\north\east\east\east\east\east\north\north\east\east;
      \draw (4,-.7) node {$\mu$};
      \draw (9,1.5) node [label=$\conj$] {$\longleftrightarrow$};
    \fill[pattern=north east lines,pattern color=green] (5,1) rectangle (6,3);
\end{scope}
\begin{scope}[shift={(10,-20)}] 
    \fill[blue!30] (0,7) rectangle (3,8);
    \fill[blue!30] (1,6) rectangle (2,7);
    \fill[red!30] (0,7) rectangle (1,2);
    \fill[red!30] (1,6) rectangle (2,2);
      \draw[gray,thin] (0,0) grid (3,8);
      \draw[very thick] (0,0) \north\north\east\east\north\north\north\north\north\east\north;
      \draw[red] (0.5,2) node [below] {$5$};
      \draw[red] (1.5,2) node [below] {$4$};
     \draw[blue] (2,6.5) node[right] {$1$};
      \draw[blue] (3,7.5) node[right] {$3$};
      \draw (-.6,.5) node {$\mu'$};
      \draw (4,4) node [label=$\phi^{-1}$] {$\longrightarrow$};
\end{scope}
\begin{scope}[shift={(16,3)},scale=.7071]
       \draw (-1,0)--(12,0);
       \draw (0,-3.5)--(0,.5);
       \draw[dashed] (-1,-2)--(12,-2);
     \draw[very thick](0,0) circle(1.2pt)\down\down\up\up\down\up\down\down\down\up\up; 
     \draw[green,very thick,dotted] (0,0) \down\down;
     \draw[green,very thick,dotted] (8,-2) \down;       
     \filldraw[purple] (2,-2) circle (.12);
      \filldraw[purple] (5,-1) circle (.12);
      \filldraw[purple] (9,-3) circle (.12);
      \draw (5,2) node {$P=\phi^{-1}(\la')$};
      \draw[->] (5.5,-5)-- node[right] {$\gamma$} (5.5,-6);
\end{scope}
\begin{scope}[shift={(16,-4)},scale=.7071]
       \draw (-1,0)--(12,0);
       \draw (0,-2.5)--(0,1.5);
       \draw[dashed] (-1,-1)--(12,-1);
     \draw[very thick](0,0) circle(1.2pt)\down\down\up\up\down\up\down\down\up\up\up; 
     \draw[green,very thick,dotted] (0,0) \down\down;
     \draw[green,very thick,dotted] (8,-2) \up;       
     \filldraw[purple] (2,-2) circle (.12);
      \filldraw[purple] (5,-1) circle (.12);
      \filldraw[purple] (8,-2) circle (.12);
      \draw[->] (5.5,-4.5)-- node[right] {$\gamma$} (5.5,-5.5);
\end{scope}
\begin{scope}[shift={(16,-11)},scale=.7071]
       \draw (-1,0)--(12,0);
       \draw (0,-1.5)--(0,3.5);
       \draw[dashed] (-1,0)--(12,0);
     \draw[very thick](0,0) circle(1.2pt)\down\up\up\up\down\up\down\down\up\up\up; 
     \draw[green,very thick,dotted] (0,0) \down\up;
     \draw[green,very thick,dotted] (8,0) \up;       
     \filldraw[purple] (1,-1) circle (.12);
      \filldraw[purple] (5,1) circle (.12);
      \filldraw[purple] (8,0) circle (.12);
      \draw[->] (5.5,-3.5)-- node[right] {$\gamma$} (5.5,-4.5);
\end{scope}
\begin{scope}[shift={(16,-18)},scale=.7071]
       \draw (-1,0)--(12,0);
       \draw (0,-0.5)--(0,5.5);
     \draw[very thick](0,0) circle(1.2pt)\up\up\up\up\down\up\down\down\up\up\up; 
      \filldraw[purple] (0,0) circle (.12);
      \filldraw[purple] (5,3) circle (.12);
      \filldraw[purple] (8,2) circle (.12);
     \draw[green,very thick,dotted] (0,0) \up\up;
     \draw[green,very thick,dotted] (8,2) \up;  
\end{scope}    
\end{tikzpicture}
  \caption{The map $f$ on partitions corresponds to the map $\gamma$ on lattice paths.
  In this example the left column agrees with
  Figure~\ref{fig:fbij_partitions}, and $\bij=f^3$ is a bijection between $\Rc_{m,n}^{\ge 1-\ell}(N)$ and $\P_{m-\ell-1,n+\ell+1}(N-\ell-1)$ where $m=6$, $n=5$, $\ell=2$, and $N=16$.
  }\label{fig:comparison}
\end{figure}
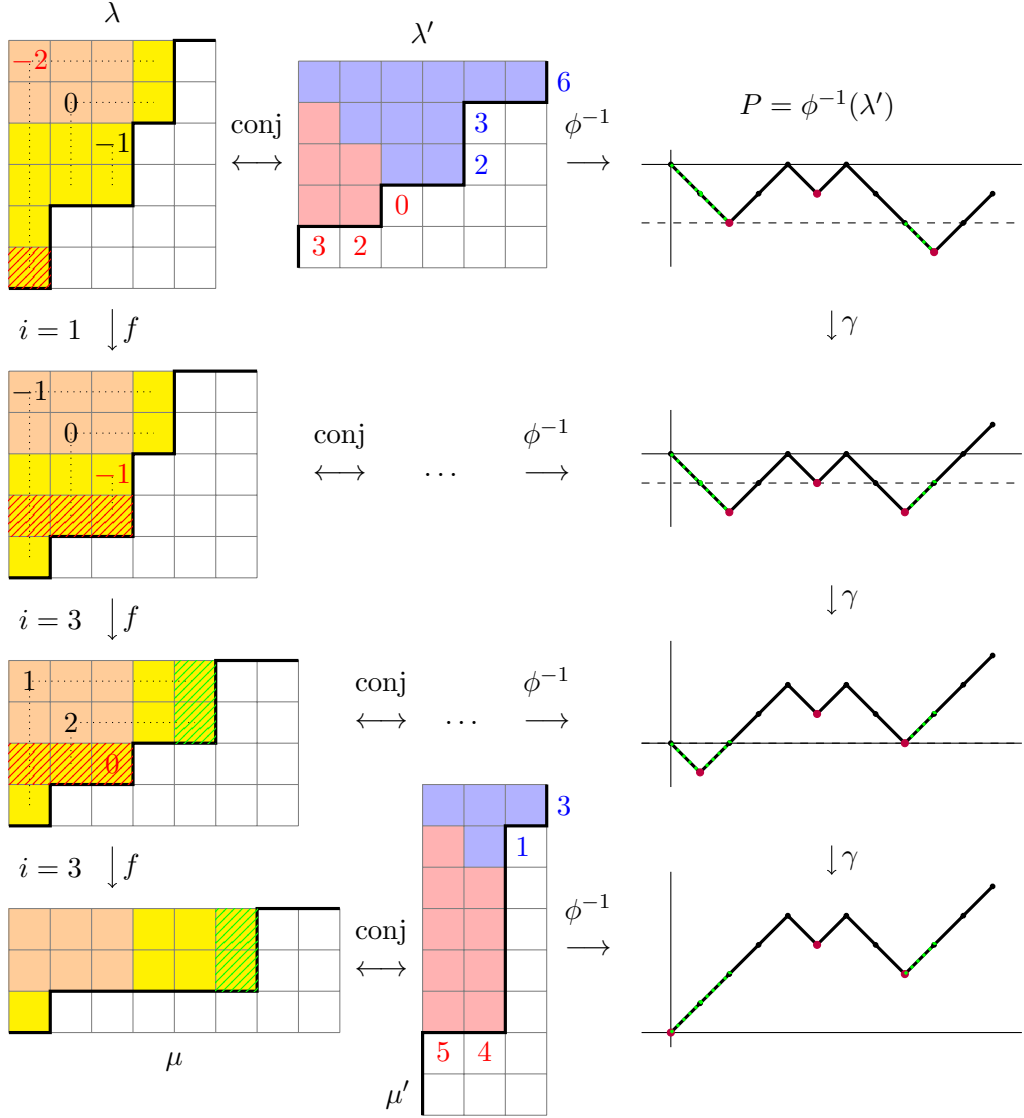

Let $P = \phi^{-1}(\la')$.  By Lemma \ref{Foata_to_hooks}, $P$ is the unique path in $\G_{m,n}$ whose $j$th valley from the left, for $1 \leq j \leq d$, is at the point
\begin{equation*}
    V_j =  (b_j+a_j,b_j-a_j).
\end{equation*}

Now let $\tau$ be the minimum rank of $\la$, and note that
$$
    \tau = \min\{b_d-a_d, b_{d-1}-a_{d-1}, \ldots, b_1-a_1\} +1.
$$
Let $i$ be the largest index such that $r_i(\la) = \tau$.  Then the valley 
$V_{d+1-i}$
is the leftmost minimum of $P$. We now show that the effect of applying $f$ to $\la$ corresponds to applying $\gamma$ to $P$.

From the proof of Lemma~\ref{f_lemma}, unless we have $\tau=0$ and $i=d=\la_d>\la_{d+1}$, the effect of applying $f$ to $\la$ is that each $\la_j$ for $1 \leq j < i$ increases by 1, and each $\la'_j$ for $1 \leq j \leq i$ decreases by 1.
Equivalently, in the hook decomposition of $\la'$, each 
$b_{d+1-j}$ for $1 \leq j < i$ increases by 1, and each
$a_{d+1-j}$ for $1 \leq j \leq i$ decreases by 1.

The changes to $\la$ and $\la'$ caused by $f$ have the following effect on the valleys of $P= \phi^{-1}(\la')$:
\begin{eqnarray*}
\text{$V_j=(b_j+a_j,b_j-a_j)$ moves to  $(b_j+a_j-1,b_j-a_j+1)$ if $j=d+1-i$;}\\
\text{$V_j=(b_j+a_j,b_j-a_j)$ moves to  $(b_j+a_j-1,b_j-a_j+2)$ if $j>d+1-i$.}
\end{eqnarray*}
This corresponds to changing the down step in $P$ leading to the valley $V_{d+1-i}$ to an up step, which is the action of $\gamma$.
This is illustrated in the first two applications of $f$ in Figure \ref{fig:comparison}. 

It remains to check the case $\tau=0$ and $i=d=\la_d>\la_{d+1}$.  
This case is illustrated in the last application of $f$ in Figure \ref{fig:comparison}. 
In this case $a_1=1$ and $b_1=0$, and thus the leftmost minimum of $\phi^{-1}(\la')$ is the valley $(1,-1)$.
The effect of applying $f$ to $\la$ is that each $\la_j$ for $1 \leq j < d-1$ increases  by 1, each $\la'_j$ for $1 \leq j \leq d$ decreases by 1,
and the side of the Durfee square decreases by 1 in both $\la$ and $\la'$.
This means that we lose the smallest hook $h_1$ in $\la'$, and the corresponding valley $V_1=(b_1+a_1,b_1-a_1)=(1,-1)$ in $P$.
As for the other valleys in $P$, $a_{d+1-j}$ decreases by 1 and $b_{d+1-j}$ increases by 1 for $1 \leq j \leq d-1$.  So, for $2 \leq j \leq d$, the valley $V_j$ moves to $V_j+(0,2)$.  This corresponds to changing the first step of $P$ (which is a $D$) to a $U$, which agrees with the action of~$\gamma$.
\end{proof}

As a consequence, we have the following alternative description of $\bij$ in terms of lattice paths.

\begin{corollary}\label{cor:bij=gammarepeat}
Let $m,n,\ell\ge0$ such that $n+\ell\ge m$. The bijection $\bij$ from Theorem~\ref{thm:bij} and the bijection $\gamma^{\ell+1}$ from Lemma~\ref{lem:GKrepeat} are related by
$$\bij(\la)=\phi(\gamma^{\ell+1}(\phi^{-1}(\la')))'$$ for $\la\in\Rc_{m,n}^{\ge 1-\ell}(N)$.
The following diagram illustrates this composition of bijections:
$$
\begin{array}{ccccc}
\Rc_{m,n}^{\ge 1-\ell}(N)&\stackrel{\conj}{\longleftrightarrow}& \Rc_{n,m}^{\le \ell-1}(N)&\stackrel{\phi^{-1}}{\longrightarrow}&
\B_{n,m}^{-\ell}(N)\medskip\\
\downarrow\footnotesize{\bij}
&&
&&\downarrow\footnotesize{\gamma^{\ell+1}}\\
\P_{m-\ell-1,n+\ell+1}(N-\ell-1)&\stackrel{\conj}{\longleftrightarrow}&\P_{n+\ell+1,m-\ell-1}(N-\ell-1)&\stackrel{\phi^{-1}}{\longrightarrow}&\G_{n+\ell+1,m-\ell-1}(N-\ell-1)\\
\end{array}
$$
\end{corollary}

See the examples in Figure~\ref{fig:comparison} illustrating this equivalence.
The correspondence in Corollary~\ref{cor:bij=gammarepeat} allows us to give an alternative proof of Theorem~\ref{thm:bij} using Lemmas~\ref{lem:GKrepeat} and~\ref{lem:Foata}.

Using Foata's correspondence, the generating functions in Theorems~\ref{thm:central-simplified} and~\ref{thm:lopsided} have interpretations in terms of lattice paths.

\section{Connections to other results on lattice paths}\label{sec:lattice_paths_other}

In this section we discuss two alternative proofs of Theorem~\ref{thm:central-simplified} using techniques from the literature on lattice paths.

\subsection{Nonintersecting paths}\label{sec:nonintersecting}

Here we prove Theorem~\ref{thm:central-simplified} using 
the Lindstr\"om--Gessel--Viennot Lemma \cite{LGV}.
Let  $\lambda\in\mathcal{P}_{m,n}$ with $d(\la)=i$.
The partition $\alpha$ to the right of the Durfee square is a partition in $\mathcal{P}_{i,n-i}$
and the partition below the Durfee square is in $\mathcal{P}_{m-i,i}$, hence its conjugate $\beta$
is in $\mathcal{P}_{i,m-i}$. Then $\lambda\in\R_{m,n}^{\ge 1-\ell}$ if and only if
$\alpha_j-\beta_j\ge 1-\ell$ for all $j$. This means that if we draw the Young diagram of $\alpha$ so that its boundary is
 a path from $A_1=(0,0)$ to $B_1=(n-i,i)$, and similarly the diagram of $\beta$ 
as a path from $A_2=(\ell,-1)$ to $B_2=(m+\ell-i,i-1)$, then $\lambda\in\R_{m,n}^{\ge 1-\ell}$ if and only if these paths do not intersect.

Consider pairs of paths from $A_1$ to $B_1$ and from $A_2$ to $B_2$, with the above notation, and define their weight to be the sum of the area above the paths. The generating function of these pairs of paths, where $q$ marks the weight, is 
$$
\qbin{n}{i}\qbin{m}{i}.
$$
The generating function of pairs of paths that intersect in at least one vertex (see Figure \ref{fig:path} for an example) is $q^\ell$ times 
the generating function of pairs of paths from $A_1$ to  $B_2$ and from $A_2$ to $B_1$, that is,
$$
q^\ell \qbin{m+\ell-1}{i-1}\qbin{n-\ell+1}{i+1}.
$$

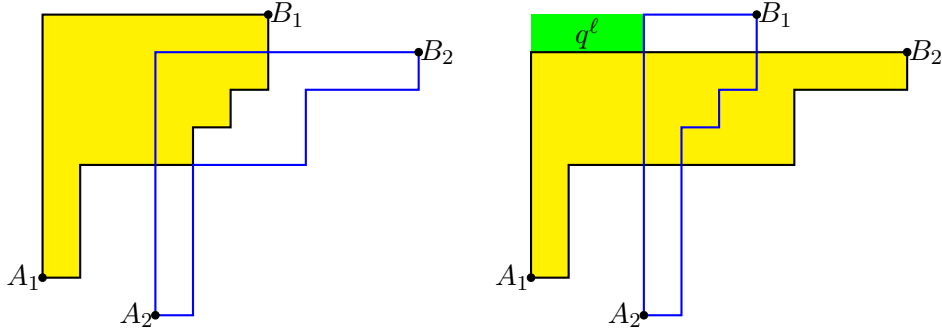
\begin{figure}[htb]
\centering
    \begin{tikzpicture}[scale=0.5]
    \fill[yellow] (0,0)--(1,0)--(1,2)--(1,3)--(4,3)--(4,4)--(5,4)--(5,5)--(6,5)--(6,7)--(0,7)--(0,0);
\draw[thick] (0,0)--(1,0)--(1,2)--(1,3)--(4,3)--(4,4)--(5,4)--(5,5)--(6,5)--(6,7)--(0,7)--(0,0);;
      \draw[blue,thick] (3,-1)--(4,-1)--(4,3)--(7,3)--(7,4)--(7,5)--(10,5)--(10,6)--(3,6)--(3,-1);
\filldraw (0,0) circle (.1);
\filldraw (6,7) circle (.1);
\draw (-0.5,0) node {$A_1$};
\draw (6.5,7) node {$B_1$};
\filldraw (3,-1) circle (.1);
\filldraw (10,6) circle (.1);
\draw (2.5,-1) node {$A_2$};
\draw (10.5,6) node {$B_2$};
    \end{tikzpicture}\hspace{.15cm}
\begin{tikzpicture}[scale=0.5]
\fill[green](0,7)--(3,7)--(3,6)--(0,6);
\fill[yellow] (0,0)--(1,0)--(1,2)--(1,3)--(4,3)--(7,3)--(7,4)--(7,5)--(10,5)--(10,6)--(0,6)--(0,0);
\draw[thick] (0,0)--(1,0)--(1,2)--(1,3)--(4,3)--(7,3)--(7,4)--(7,5)--(10,5)--(10,6)--(0,6)--(0,0);
\draw[blue,thick] (3,-1)--(4,-1)--(4,3)--(4,4)--(5,4)--(5,5)--(6,5)--(6,7)--(3,7)--(3,-1);
\filldraw (0,0) circle (.1);
\filldraw (6,7) circle (.1);
\draw (-0.5,0) node {$A_1$};
\draw (6.5,7) node {$B_1$};
\filldraw (3,-1) circle (.1);
\filldraw (10,6) circle (.1);
\draw (2.5,-1) node {$A_2$};
\draw (10.5,6) node {$B_2$};
\draw (1.5,6.5) node {$q^\ell$};
 \end{tikzpicture}
  \caption{Example of a pair of paths that intersect.}\label{fig:path}
\end{figure}

Thus, the generating function for pairs of paths that do not intersect is
$$
\qbin{n}{i}\qbin{m}{i}-q^{\ell}\qbin{m+\ell-1}{i-1}\qbin{n-\ell+1}{i+1}.
$$

Multiplying by $t^i q^{i^2}$ to account for the contribution of the Durfee square, and summing over $i$, we recover the right-hand side of Theorem~\ref{thm:central-simplified}.

\subsection{Paths between two lines}\label{sec:refined}

In terms of lattice paths, Theorem~\ref{thm:central-simplified} states that, 
for $m,n,\ell\ge0$ such that $n+\ell\ge m$, 
\begin{equation}\label{eq:central-lattice-paths}
\sum_{P\in\A^{-\ell}_{n,m}}t^{\des(P)}q^{\maj(P)}=
\sum_{i\ge0} t^i q^{i^2}\left(\qbin{n}{i}\qbin{m}{i}-q^\ell\qbin{n+\ell-1}{i-1}\qbin{m-\ell+1}{i+1}\right).
\end{equation}

Note that this formula generalizes Proposition~\ref{prop:FH}. Another method to prove this formula comes from Krattenthaler and Mohanty's results~\cite{KM} on the enumeration of lattice paths that lie between two lines with respect to the number of peaks and the sum of the positions of the peaks. Equation~\eqref{eq:central-lattice-paths} is obtained when reflecting the paths and removing one of the boundary lines. The proofs in~\cite{KM} are based on inclusion-exclusion on the pairs of sequences obtained by recording the coordinates of the peaks, which is different but plays a similar role to the Greene--Kleitman mapping in our proof.

There is also an analogue to equation~\eqref{eq:central-lattice-paths} where one considers peaks instead of valleys.
For $P\in\G_{n,m}$, let $\hdes(P)$ denote the number of peaks of $P$, and $\hmaj(P)$ the sum of the positions of the peaks.
Some literature on standard Young tableaux relates to the enumeration of paths in $\A^{-\ell}_{n,m}$ with respect to these statistics. In particular, by interpreting $\hdes$ and $\hmaj$ as the number of descents and major index of $2$-row skew standard Young tableaux, Keith \cite[Thm.\ 4]{Keith} proves the following result, which generalizes the case $\ell=0$ that was already proved in \cite[Cor.\ 15]{BBES16} in an equivalent form. As noted in~\cite{Keith}, this result is also a special case of \cite[Thm.\ 1]{KM}.

\begin{theorem}[\cite{Keith,KM}]
$$\sum_{P\in\A^{-\ell}_{n,m}}t^{\hdes(P)}q^{\hmaj(P)}=\sum_{i\ge0} t^iq^{i^2}\left(\qbin{n}{i}\qbin{m}{i}-\qbin{n+\ell+1}{i}\qbin{m-\ell-1}{i}\right).$$
\end{theorem}

Next we give an alternative proof of this theorem using the Greene--Kleitman mapping. This is significantly simpler than Keith's proof, and different from Krattenthaler and Mohanty's proof.

\begin{proof}
Reflecting along the $x$-axis, we can use Lemma~\ref{lem:qbin2} to enumerate unconstrained paths with respect to $\hdes$ and $\hmaj$:
\begin{equation}\label{eq:hdeshmajG}
    \sum_{P\in\G_{n,m}}t^{\hdes(P)}q^{\hmaj(P)}=\sum_{P\in\G_{m,n}}t^{\des(P)}q^{\maj(P)}=\sum_{i\ge0} t^i q^{i^2}\qbin{n}{i}\qbin{m}{i}.
\end{equation}
Next we subtract the contribution of paths in $\B^{-\ell}_{n,m}$. It is clear that the bijection $\gamma$ from Definition~\ref{def:gamma} preserves the number and location of the peaks of the path, and hence the statistics $\hdes$ and $\hmaj$. Thus, these statistics are preserved by the bijection $\gamma^{\ell+1}:\B^{-\ell}_{n,m}\to\G_{n+\ell+1,m-\ell-1}$. It follows that 
\begin{equation}\label{eq:hdeshmajB}
    \sum_{P\in\B^{-\ell}_{n,m}}t^{\hdes(P)}q^{\hmaj(P)}=\sum_{P\in\G_{n+\ell+1,m-\ell-1}}t^{\hdes(P)}q^{\hmaj(P)}=\sum_{i\ge0} t^i q^{i^2}\qbin{n+\ell+1}{i}\qbin{m-\ell-1}{i}.
\end{equation}
Subtracting equation~\eqref{eq:hdeshmajB} from equation~\eqref{eq:hdeshmajG}, we obtain the stated formula.
\end{proof}

\section{Other restrictions on ranks}
\label{sec:final}

In this paper we have considered restrictions on ranks by bounding the allowed ranks from below (or equivalently, by conjugation, from above). It is possible to use equation~\eqref{eq:phiS} to study partitions with other restrictions on ranks. Next we consider a few examples where the ranks are constrained to be in a finite set.

\begin{example}
Partitions in $\R^{\{0\}}$ are self-conjugate partitions. Equation~\eqref{eq:phiS} gives a bijection $\R^{\{0\}}_{m,n}\to\V_{m,n}^{\{-1\}}$.
Since paths in $\V_{m,n}^{\{-1\}}$ are determined by the positions of their valleys, they are in bijection with partitions into distinct odd parts, with largest part $\le 2\min(m,n)-1$.
When $m,n\rightarrow\infty$,
the composition of these two bijections  yields a well-known bijection between self-conjugate partitions and partitions into  distinct odd parts.
\end{example}

\begin{example} 
Taking $S=\{0,-1\}$, equation~\eqref{eq:phiS} gives a bijection
$\R^{\{0,-1\}}_{m,n}\to\V_{m,n}^{\{0,-1\}}$.
Again, by recording the positions of their valleys, these paths are in bijection with partitions  into parts that differ by at least $2$, with largest part $\le M$, where
$$M=\begin{cases} 2m-1 &  \text{if }m\le n,\\
2n & \text{otherwise}.\end{cases}$$
The composition of the two bijections preserves the area of the partition, and it sends the side of the Durfee square to the number of parts.

It follows that 
$$ \sum_{\la\in \R^{\{0,-1\}}_{m,n}} t^{d(\la)}q^{|\la|}=\sum_{k=0}^n t^k q^{k^2} \qbin{M-k+1}{k}. $$
Letting $n,m\to\infty$, we get
\begin{equation}
\sum_{\la\in \R^{\{0,-1\}}} t^{d(\la)}q^{|\la|}=\sum_{k\ge 0} \frac{t^k q^{k^2}}{(1-q)(1-q^2)\dots(1-q^k)},
\label{RR1}
\end{equation}
recovering a classical result of Bressoud~\cite{Bressoud}.
\end{example}

\begin{example} 
Taking $S=\{-1,-2\}$, we obtain a bijection $\R^{\{-1,-2\}}_{m,n}\to\V_{m,n}^{\{0,1\}}$.
These paths are in bijection with partitions into parts that differ by at least $2$, with smallest part $\ge 2$ and largest part at most $M$, where
$$M=\begin{cases} 2m-2 & \text{if }m\le n+1,\\
2n+1 & \text{otherwise}.\end{cases}$$ 
It follows that, for $m,n\ge2$,
$$\sum_{\la\in \R^{\{-1,-2\}}_{m,n}} t^{d(\la)}q^{|\la|}=\sum_{k=0}^{n-1} t^k q^{k(k+1)} \qbin{M-k}{k}.$$
\end{example}

\begin{example} 
Taking $S=\{0,-2\}$, we obtain a bijection $\R^{\{0,-2\}}_n\to\V_n^{\{-1,1\}}$. In this case, such paths are not uniquely determined by the positions of their valleys. Instead, a valley at a given position can be at either of the two available heights, with the exception that valleys whose positions differ by $2$ are forced to be a the same height, and valleys in positions $1$ and $2n-1$ must be at height $-1$.
Defining a {\em 2-block} to be a maximal arithmetic progression of difference two, paths in $\V_n^{\{1,-1\}}$ are in bijection with partitions into odd distinct parts, with largest part $\le 2n-1$, where each 2-block of parts can be colored in one of two colors, except if the block contains part $1$ or part $2n-1$. Partitions that have $k$ parts, after subtracting $2(k-i)+1$ from the $i$th largest part, correspond to partitions into at most $k$ even parts, with largest part $\le 2(n-k)$, and where each block of equal parts (other that $2(n-k)$) gets one of two colors.
Letting $n\to\infty$, it follows that 
$$\sum_{\la\in \R^{\{0,-2\}}} t^{d(\la)}q^{|\la|}=\sum_{k\ge0} \frac{t^k q^{k^2}(1+q^2)(1+q^{4})\dots(1+q^{2k})}{(1-q^2)(1-q^4)\dots(1-q^{2k})}.$$

More generally, for any positive integer $a$, we get
$$\sum_{\la\in \R^{\{0,-a\}}} t^{d(\la)}q^{|\la|}=\sum_{k\ge0} \frac{t^k q^{k^2}(1+q^a)(1+q^{2a})\dots(1+q^{ak})}{(1-q^2)(1-q^4)\dots(1-q^{2k})}.$$
For $a=1$, we recover equation~\eqref{RR1}. 
\end{example}

For arbitrary sets $S$, we do not have a general method to derive a formula for the generating function in equation~\eqref{eq:RV}, even in the case $m,n\to\infty$.
At the level of paths, the difficulty arises when trying to keep track of the statistic $\maj$ (which corresponds to the area of the partition). However, if we disregard this parameter and only keep track of $\des$ (which corresponds to the side of the Durfee square), then it is possible to use continued fractions to enumerate paths whose valleys can occur only at certain arbitrary heights.
More generally, one can count paths with respect to the number of valleys whose heights lie inside or outside a certain set, and deduce generating functions of the form
$$\sum_{n\ge0}\sum_{\la\in\P_n}t^{|\{i:r_i(\la)\in S\}|}u^{|\{i:r_i(\la)\notin S\}|}
z^n$$
for arbitrary $S$.

For example, letting $S$ be the set of odd integers, we can enumerate partitions with respect to the number of odd and even ranks:
\begin{multline*}\sum_{n\ge0}\sum_{\la\in\P_n}t^{|\{i:r_i(\la)\text{ odd}\}|}u^{|\{i:r_i(\la)\text{ even}\}|}
z^n
=\sum_{n\ge0}\sum_{P\in\G_n}t^{|\{i:v_i(P)\text{ even}\}|}u^{|\{i:v_i(P)\text{ odd}\}|}z^n\\
=\frac{1-(t-1)z}{\sqrt{(1-(t-1)z)(1-(u-1)z)(1+(u-1)(t-1)z^2-(t+u+2)z)}}.
\end{multline*}

\subsection*{Acknowledments} SC is partially funded by NSF
grant DMS-2054482 and by grant ANR COMBINE ANR-19-CE48-0011. SE was partially supported by Simons Collaboration Grant \#929653.

\end{document}